\documentclass[a4paper,12pt,english,reqno]{amsart}
\usepackage[latin1]{inputenc}
\usepackage{amsmath}
\usepackage{amssymb}
\usepackage{url}
\usepackage{graphicx}
\usepackage{color}
\usepackage{dsfont}
\usepackage{a4wide}
\usepackage{multicol}
\usepackage{varioref}
\usepackage{marvosym}
\usepackage[useregional]{}
\numberwithin{equation}{section}

\usepackage[colorlinks=true , linkcolor=blue]{hyperref}
\hypersetup{urlcolor=red, citecolor=magenta}
\newtheorem{th1}{Theorem}[section]
\newtheorem{theorem}[th1]{Theorem}
\newtheorem{lemma}[th1]{Lemma}
\newtheorem{proposition}[th1]{Proposition}

\newtheorem{remark}[th1]{Remark}
\newtheorem{corollary}[th1]{ Corollary}

\title[Indirect stabilization of semilinear coupled wave systems]{Indirect stabilization of semilinear coupled wave systems}

\title[Indirect stabilization of semilinear coupled wave systems]
{Indirect stabilization of semilinear coupled wave systems} 

\author{Radhia \textsc{Ayechi}}
\address{University of Sousse, ESSTHS, LAMMDA, Tunisia}
\email{radhiaayechi@essths.u-sousse.tn}

\author{Moez \textsc{Khenissi}}
\address{University of Sousse, ESSTHS, LAMMDA, Tunisia}
\email{moez.khenissi@essths.u-sousse.tn}

\author{Camille \textsc{Laurent}}
\address{CNRS UMR 7598 \& Sorbonne Universit\'e \\
Laboratoire Jacques-Louis Lions, F-75005, Paris, France}
\email{camille.laurent@sorbonne-universite.fr}


\begin{document}
	\maketitle


	
	\begin{abstract} In this paper, we study the indirect stabilization problem for a system of two coupled semilinear wave equations with internal damping in a bounded domain in $\mathbb{R} ^3$. The nonlinearity is assumed to be subcritical, defocusing and analytic. Under geometric control condition on both coupling and damping regions, we establish the exponential energy decay rate.
	\end{abstract}

	\section{Introduction}
	
	This paper is devoted to the study of the following semilinear coupled wave system in a bounded domain of $ \mathbb{R}^3$ with a smooth boundary  $ \Gamma = \partial \Omega $:
	\begin{equation}\label{system}
	\begin{cases}
	\partial _{tt} u - \Delta u  + a(x) \partial _t u  + b(x) \partial _t v + f_1(u)= 0 & in \; \Omega \times \mathbb{R}_{+} ^{*} , \\
	\partial _{tt}	v  -\Delta v  - b(x) \partial _t u +f_2(v)= 0 & in \; \Omega \times \mathbb{R}_{+} ^{*} ,  \\
	u=v=0  & on \; \Gamma \times \mathbb{R}_{+} ^{*} , \\
	u(x,0)=u_{0} (x)\; , \;  v(x,0)=v_{0}(x)  & in \; \Omega ,  \\
	\partial _t	u (x,0) = u_{1}(x) \; , \; \partial _t v(x,0) = v_{1}(x) & in \; \Omega ,
	\end{cases}
	\end{equation}
	where the damping term $a \in L^\infty (\Omega)$ is a non-negative function, the coupling term \\ $b \in L^\infty (\Omega)$ is non-negative and the initial data $(u_{0},v_{0},u_{1},v_{1})$ is in the energy space\\ $\mathcal{H}:= \left( H_0^1(\Omega) \right)^2 \times\left( L^2(\Omega) \right)^2$. We denote by $ \Delta $ the Laplace operator on $ \Omega $.\\
	The non-linearity $f_i \in \mathcal{C}^1(\mathbb{R},\mathbb{R})$, for $i=1,2$, is assumed to be defocusing, energy subcritical and such that $0$ is an equilibrium point. More precisely, we assume that there exists $C>0$ such that 
	\begin{eqnarray}\label{nonlinearity}
	f_{i}(0) =0, \; sf_{i}(s)\geq 0 , \; \vert f_{i}(s) \vert \leq C (1+\vert s \vert )^p  \; \text{and } \vert f_{i}^\prime (s) \vert  \leq C (1+\vert s \vert )^{p-1} ,
	\end{eqnarray}
	with $1\leq p < 5$.

	We will check that problem \eqref{system} is well posed. Then the associated energy $E_{u,v} $ of a solution $(u,v)$ at time $t$ is defined by:
	\begin{eqnarray}
	E_{u,v}(t):= E(u,v, \partial _t u ,\partial _t v) & = & \frac{1}{2} \int_{\Omega} \left(   \vert \triangledown u(x,t) \vert ^{2} + \vert \partial _t u (x,t) \vert ^{2}  + \vert \triangledown v(x,t) \vert ^{2} + \vert \partial _t v (x,t) \vert ^{2} \right)    dx \nonumber \\
	&&	+ \int_{\Omega}  \mathcal{G}\left( u(x,t),v(x,t)\right) dx , 
	\end{eqnarray}
	where $ \mathcal{G}(u,v)= \displaystyle \int_{0}^{u} f_1(s) ds + \displaystyle \int_{0}^{v} f_2(s) ds$. 
	A straightforward computation shows that this energy is non-increasing: 
	\begin{eqnarray}
	E ^\prime _{u,v}  (t) = - \int _{\Omega} a(x) \vert \partial _t u (x,t) \vert ^{2} dx  \leq 0,
	\end{eqnarray}
	and  system \eqref{system} is therefore dissipative. Due to assumption \eqref{nonlinearity} and the Sobolev embedding $H ^1 (\Omega) \hookrightarrow L^6 (\Omega) $, for all $E_0 \geq 0 $ there exists $C\geq 0$ such that 
	\begin{eqnarray}\label{equiv}
	(u,v , \tilde{u} , \tilde{v}) \in  \mathcal{H},  \text{ with } E(u,v,\tilde{u},\tilde{v}) \leq E_0 \; \;\; \Rightarrow \frac{1}{C} \Vert (u,v,\tilde{u} ,\tilde{v}) \Vert _{\mathcal{H}} ^{2} \leq E(u,v,\tilde{u},\tilde{v}) \leq C \Vert (u,v,\tilde{u},\tilde{v} )\Vert _{\mathcal{H}} ^{2}.\quad 
	\end{eqnarray}
	The lower bound is a consequence of the positivity of $ \mathcal{G}(u,v)$ thanks to \eqref{nonlinearity}.
	The aim of this paper is to give sufficient conditions on the non-linearity and on both coupling and damping regions, ensuring the uniform exponential decay of the energy. More precisely, for some positive constants $a_0$ and $b_0$, we assume that  $\omega_{a}$ and $\omega_{b}$ are two open subsets of $\Omega$, so that
	\begin{itemize}
		\item  $Supp(b)\subset \omega _a \subset \left\lbrace  a(x) \geq a_0 > 0 \right\rbrace  $ 
		\item $ \omega _b \subset  \left\lbrace   b(x)  \geq b_0 > 0 \right\rbrace  $
		\item $\omega_{b}$  satisfies the geometric control condition \footnote{	 We say that $ \omega_b $ satisfies the geometric control condition if every generalized geodesic (i.e. ray of geometric optics) traveling at speed one in $ \Omega $ meets $ \omega_b $ in a uniform time.}(see \cite{1}).
	\end{itemize}
	Note that the assumptions imply $\omega _b \subset\omega_{a}$, where both particularly satisfy the geometric control condition. The classical results of Bardos-Lebeau-Rauch (see \cite{1,24}) imply that for the scalar equation with damping $a\partial_{t}$, the energy decays exponentially. It should be  recalled that $b$ is assumed nonnegative. Yet, replacing $v$ by $-v$ changes $b$ to $-b$ in equation \eqref{system}. In particular, our result remains true in case $b \leq 0$, assuming that $ \omega _b \subset  \left\lbrace -  b(x)  \geq b_0 > 0 \right\rbrace  $. We will stick to the case $b\geq 0$ in what follows to avoid confusion. Our main result is the following:
	\bigskip
	\begin{theorem}\label{th1}
		We assume that the open sets $\omega_{a}$ and $ \omega_{b} $ satisfy the previous assumptions. If $f_{1}$ and $f_{2}$ are real analytic and satisfy \eqref{nonlinearity}, then for any $E_0 \geq 0$, there exist $ C >0  $ and $ \beta >0  $ such that, for all solutions $ (u,v) $ of system $ \eqref{system}$ with $ E_{u,v}(0)\leq E_0$,
		\begin{eqnarray}\label{energ}
		\forall t\geq 0, \quad	E_{u,v}(t) \leq C e^{-\beta t} E_0 . 
		\end{eqnarray}
	\end{theorem}
	This result means that the damping "$a  \partial _t u $", applied in only one equation for the system, stabilizes any solution of \eqref{system} to zero, which is an important property from the dynamical and control point of view.\\
	The proof of this result is mainly of the type: "geometric control condition" + "unique continuation" implies "exponential decay".

	Concerning  the problem of stabilization of a linear damped wave equation, uniform exponential decay has been obtained in Bardos-Lebeau-Rauch \cite{1} and Lebeau \cite{24} under the usual geometric control condition (\textbf{GCC}). Roughly speaking, the assumption is that every ray of geometric optics enters the damping region in a uniform time. The geometric control condition is known to be not only sufficient but also necessary for the exponential decay of the linear damped equation. Note that a large amount of other results have been obtained in this context for different questions: obtaining different  decay rates under weaker geometric assumptions (see for instance \cite{anantharaman2014sharp,22} and the  references therein) and studying the phenomenon induced by the infinite domain \cite{aloui2015energy,bouclet2014local,joly2018energy,khenissi2003equation,cavalcanti2008uniform,fatiha}. 
	
	The research about linear partially damped wave systems seems more recent, but the subject has been very active. There has been several results using the functional analytic method as the one of Alabau \cite{A:15} and Alabau-L\'eautaud \cite{fatiha} (see also the survey  \cite{A:12}). We can also quote Klein \cite{K:18}, who obtained a general result with a matricial condition using microlocal analysis methods, and Cui-Wang \cite{CW:21} for a more specific problem. In this article, we will use the linear result of Ayechi-Khenissi \cite{radhia2022local} where the authors established the uniform stability when the coupling region is contained in the damping region and satisfied the geometric control condition. It is worth noting  also that the question of the damping for a system of waves is closely related to the problems of controllability for systems for which there has been recent progress. We refer for instance to Alabau and L\'eautaud  \cite{A:03,AB:13} using functional analysis or Dehman-L\'eautaud- Le Rousseau \cite{DLL:14} with microlocal methods. Moreover, in the context of a compact manifold, the general result of Cui-Laurent-Wang \cite{cui2020observability} proved the equivalence of observability with a system of ODE along the rays of geometric optics. It can be applied to prove the expected observability estimate for the model system of this article, but in the case without boundary. We refer to Remark \ref{rklienODE} for more precision with the link with a natural ODE system along the bicharacteristic flow.

	The case of a single semi-linear wave equation was studied in \cite{haraux1985stabilization, zuazua1990exponential, zuazua1991exponential, dehman2003stabilization} for $p<3$. The first result for $ p \in \left[ 3, 5 \right) $ was the one of Dehman, Lebeau and Zuazua \cite{dehman2003stabilization}. This work was mainly concerned with the stabilization problem on the Euclidean space $\mathbb{R}^3$ with flat metric and stabilization active outside of a ball. Nevertheless, it settled a general scheme of proof for the stabilization of a nonlinear subcritical problem of the form: "propagation of compactness " + "unique continuation" implies "observability". However, the problem of the unique continuation problem under a general geometric control condition has not been solved in general. This problem was solved in Joly-Laurent \cite{2} under the assumption of the analyticity of the nonlinearity. We will follow this scheme of proof concerning the nonlinear part of the proof.
	
	Other stabilization results for the nonlinear wave equation can be found in Aloui-Ibrahim-Nakanishi \cite{aloui2010exponential} with a stronger geometric condition, but a large class of nonlinearity. We also refer to Cavalcanti-Cavalcanti-Fukuoka-Pampu-Astudillo \cite{CCFPA:18} and the references therein for some related equations with nonlinear damping. Some works have been done in the difficult critical case $p=5$; we can refer to \cite{dehman2002stabilization, laurent2011stabilization}. 
	
	\medskip
	
	We begin this paper by proving the global existence and uniqueness results in section \ref{s:exis}. In section \ref{s:stab}, under geometric control condition on both damping and coupling terms, we first prove a compactness result and then a unique continuation result which ultimately will prove the exponential stabilization for the linear coupled system and the uniform stabilization result.

	In the following, the norm on $L^2(\Omega)$ is written $\Vert . \Vert $ and we write $\left\langle  . \right\rangle $ for the scalar product on $L^2(\Omega)$.  Furthermore, $C$ will denote some constant (that could depend on the fixed parameters of the problem: $\Omega$, $a$, $b$, $f$) with a value that can change from one line to another.
	
	\section{Global existence and uniqueness}
	\label{s:exis}
	In this section, we prove the existence and uniqueness result of the semilinear coupled wave system \eqref{system} in the energy space. Then we reduce system \eqref{system} to a Cauchy problem: 
	\begin{eqnarray}\label{2}
	\begin{cases}
	\mathcal{U} ^\prime  (t) = \mathcal{A} \mathcal{U}(t) + \mathcal{F} ( \mathcal{U}), \\
	\mathcal{U}(0)=\mathcal{U}_{0} \in \mathcal{H},
	\end{cases}
	\end{eqnarray}
	where $\mathcal{U}=(	u ,v ,\partial _t u , \partial _t v)^T  $, $\mathcal{U}_0=(	u_0 ,v _0,u_{1} ,v_{1})^T  $,  $ \mathcal{F}(\mathcal{U}) =(0,	0,-f_1(u) ,-f_2(v))^T$ and the operator $ \mathcal{A} $ is given by 
	$$  \mathcal{A}=
	\begin{pmatrix}
	0 & 0 & I & 0 \\
	0 & 0 & 0 & I \\
	\Delta & 0 &  -a(x) & -b(x) \\
	0& \Delta & b(x)  & 0
	\end{pmatrix}
	$$ 
	with domain $ \mathcal{D}(\mathcal{A})= \left( H_{0} ^{1} ( \Omega ) \cap H^{2}( \Omega ) \right) ^{2} \times \left( H_{0} ^{1} ( \Omega ) \right) ^{2}  $. \\
	It is clear that $\mathcal{A}$ is a maximal dissipative operator on the Hilbert space $\mathcal{H}$ (\cite{radhia2022local}). Thus, by the Lummer-Philips theorem (see \cite{3}), it generates a $\mathcal{C}_0$ semigroup of contraction $(\mathcal{S}(t))_{t\geq 0}$ on the Hilbert space $\mathcal{H}$. In the following,  we note:
	\begin{eqnarray}
	\Vert u \Vert _{L^{q}_{T}L^{r}}:= \left[ \int _{0}^{T} \Vert u(t,.) \Vert _{L^{r}( \Omega )} ^{q}dt \right] ^{\frac{1}{q}} \; \; , \text{for } T>0, \; q,r \in \left[ 1, +\infty \right]  .\nonumber \\ \nonumber
	\end{eqnarray}
	In this section, we will prove the following result: 
	\begin{theorem}{Cauchy problem}\label{theorem}\item 
		For any $ \mathcal{U}_0 \in \mathcal{H} $ there exists a unique  solution $(u,v) \in \left( \mathcal{C} (\mathbb{R}_+ , H_0 ^1(\Omega)) \cap\mathcal{C}^1 (\mathbb{R}_+ , L ^2(\Omega))\right)^2 $ 
		of the semilinear coupled wave system \eqref{system}. Moreover, this solution satisfies the following Strichartz estimate: for every finite $T>0$ and $(q,r) $ where 
		\begin{eqnarray}\label{coupleadm}
		\frac{1}{q} + \frac{3}{r} = \frac{1}{2} \quad , q \in \left[ \frac{7}{2} , +\infty \right]  
		\end{eqnarray}
		there exists a constant $C=C  (\Vert \mathcal{U}_0 \Vert_{\mathcal{H}})>0$ such that :
		\begin{eqnarray}\label{strichartz}
		\Vert (u,v) \Vert _{\left( L^q(\left[ 0,T \right] ,L^r(\Omega))\right) ^2} \leq C  (\Vert \mathcal{U}_0 \Vert_{\mathcal{H}})  .
		\end{eqnarray}
		Moreover, the solution is unique in  $\left( \mathcal{C} (\mathbb{R}_+ , H_0 ^1(\Omega)) \cap\mathcal{C}^1 (\mathbb{R}_+ , L ^2(\Omega))\cap L^{q}_{T}L^{r}\right) ^2$ with suitable $q$ and $r$ as in \eqref{coupleadm}. 
	\end{theorem}
	The couple $(q,r)$ will be chosen according to the nonlinearity, see for instance \eqref{qr} for the explicit choice. In this theorem and in all the article, the solutions of \eqref{system} are meant to be as  the usual Duhamel formulation, see below. It implies to be a solution in the distributional sense. 
	
	We prove this theorem by the fixed point method with a result of the Strichartz estimate  for the linear coupled wave system. This result is given by the following theorem.
	\vspace{1cm}
	
	

	\begin{theorem}{Strichartz estimate} \label{th2.3}\item 
		Let $T> 0$ and $(q,r)$ satisfy \eqref{coupleadm}.
		There exists $C=C(T,q)>0$	such that for every $ G_1 , \; G_2 \in L^1\left( \left[ 0,T \right] ,L^2(\Omega ) \right) $ and every $(u_0,v_0, u_1 ,v_1) \in \mathcal{H}  $, the solution $ (u,v) $ of the following linear coupled wave system:
		\begin{eqnarray}\label{termesource}
		\begin{cases}
		\partial _{tt}	u(x,t) - \Delta u (x,t) + a(x) \partial _t u (x,t)+ b(x) \partial _t v (x,t) = G_1(t) & in \; \Omega \times \mathbb{R}_{+} ^{*} , \\
		\partial _{tt}	v(x,t) - \Delta v (x,t) - b(x) \partial _t u (x,t) = G_2(t) & in \; \Omega \times \mathbb{R}_{+} ^{*} , \\
		u(x,t)=v(x,t)=0  & on \; \Gamma \times \mathbb{R}_{+} ^{*} , \\
		u(x,0)=u_{0} (x)\; , \;  \partial _t u (x,0)=u_1(x)  & in \; \Omega,\\
		v(x,0)=v_{0} (x)\; , \;  \partial _t v (x,0)=v_1(x)  & in \; \Omega,
		\end{cases}
		\end{eqnarray}
		satisfies the estimate
		\begin{eqnarray}\label{strichartzesti}
		\Vert (u,v) \Vert _{ (L^q(\left[ 0,T \right] ,L^r(\Omega)))^2} \leq C \left(\Vert \mathcal{U}_0 \Vert_{\mathcal{H}}  + \Vert (G_1,G_2) \Vert_{ (L^1(\left[ 0,T \right], L^{2}(\Omega)))^2} \right) .
		\end{eqnarray}
	\end{theorem}
	The central argument to prove the previous theorem is the use of the Strichartz estimate for a non-homogeneous damped wave equation given in \cite[Corollary 1.2]{4}. The last result was proved by Burq, Lebeau and Planchon (see \cite{5}) for $q \in \left[ 5, + \infty \right] $ and extended to a larger range by Blair, Smith and Sogge in \cite{4}. We note that the Strichartz estimate was first established in the Euclidean space $\mathbb{R} ^3 $ by Strichartz \cite{6} and Ginibre and Velo (\cite{7}, \cite{8}) with $ q \in \left[ 2,+\infty \right] $.

	We note that the associated energy $E $ of a solution $(u,v)$ of system \eqref{termesource} at time $t$ is defined by:
	\begin{eqnarray}
	E(t) & = & \frac{1}{2} \int_{\Omega} \left(   \vert \triangledown u(x,t) \vert ^{2} + \vert \partial _t u (x,t) \vert ^{2}  + \vert \triangledown v(x,t) \vert ^{2} + \vert \partial _t v (x,t) \vert ^{2} \right)    dx . \nonumber
	\end{eqnarray}
	In order to prove the previous theorem, we need the following energy estimate.
	\begin{lemma}  \label{Energies identities}  There exists $C>0$ such that for  every $T> 0$, $ G_1 , \; G_2 \in L^1\left( \left[ 0,T \right] ,L^2(\Omega ) \right) $ and every $ \mathcal{U}_0 =(u_0,v_0, u_1 ,v_1) \in \mathcal{H}  $, the energy $E$ of the solution $ (u,v) $ of system \eqref{termesource} satisfies
		\begin{eqnarray}\label{RS}
		\sqrt{E(t)} \leq C \left( \Vert \mathcal{U}_0 \Vert_{\mathcal{H}}  + \Vert (G_1,G_2) \Vert_{ (L^1(\left[ 0,T \right], L^{2}(\Omega)))^2} \right)  \quad \text{for}\; t\in  \left[ 0, T \right] .
		\end{eqnarray}
	\end{lemma}
	\begin{proof}
		We note $ G= (0,0,G_1, G_2)$ and $\mathcal{U}=(u,v,\partial_t u,\partial_t v)$ such that $(u,v)$ is the solution of the system \eqref{termesource}. Then, for all $t\in  \left[ 0, T \right] $ we have, from the Duhamel formula,
		\begin{eqnarray}\label{S0}
		\mathcal{U}(t)= \mathcal{S}(t)\mathcal{U}_0 + \mathcal{W} (t), 
		\end{eqnarray}
		where $\mathcal{W} (t) = \displaystyle\int_{0}^{t} \mathcal{S}(t-s) G(s) ds $. Since $ (\mathcal{S}(t))_{t\geq 0} $ is a semi group of contractions, then using \eqref{equiv}, we conclude 
		\begin{eqnarray}\label{PPhiUI}
		\Vert  \mathcal{U} (t)\Vert_{\mathcal{H} } & \leq & \Vert \mathcal{S}(t)\mathcal{U}_0 \Vert_{\mathcal{H} } +  \Vert \mathcal{W}(t) \Vert_{\mathcal{H} }, \nonumber \\
		& \leq &  \Vert \mathcal{U}_0 \Vert_{\mathcal{H} } +  \Vert  G \Vert _{L^1_T(\mathcal{H})}. \nonumber 
		\end{eqnarray}

	\end{proof}
	\begin{proof}[Proof of Theorem \ref{th2.3}]
		For all $(q,r) $ satisfying \eqref{coupleadm} and using the Strichartz estimate for non-homogeneous damped wave equation given in  \cite[Theorem 2.1]{2} (see also \cite[Corollary 1.2]{4}), we get 
		\begin{eqnarray}
		\begin{cases}
		\Vert u \Vert _{ L_{T}^{q} L^{r}} \leq C \left(  \Vert u_{0} \Vert _{ H^{1} } + \Vert u_{1} \Vert _{ L^{2} } + \Vert b \partial _t v \Vert _{ L_{T}^{1} L^{2}} + \Vert G_1 \Vert _{ L_{T}^{1} L^{2}}\right) ,  \\
		\Vert v \Vert _{ L_{T}^{q} L^{r}} \leq C \left(   \Vert v_{0} \Vert _{ H^{1} } + \Vert v_{1} \Vert _{ L^{2} } + \Vert b \partial _t u \Vert _{ L_{T}^{1} L^{2}}  + \Vert G_2 \Vert _{ L_{T}^{1} L^{2}} \right)  .
		\end{cases}\nonumber
		\end{eqnarray}
		Now, using Lemma \ref{Energies identities} and the fact that $b \in L ^ \infty (\Omega) $, there exists $C=C(T)>0$ such that
		$$	\Vert b \partial _t u \Vert_{ L_{T}^{1} L^{2}} = \int_{0}^{T}\Vert b \partial _t u (t)\Vert_{L^{2}(\Omega)} dt \leq C \int_{0}^{T}\sqrt{E(t)} dt \leq  C \left(\Vert \mathcal{U}_0 \Vert_{\mathcal{H}}  + \Vert (G_1,G_2) \Vert_{ (L^1_{T}L^{2})^2} \right) .$$
		Finally, we get 
		$$ \Vert (u,v) \Vert _{ (L^q_{T},L^r))^2} \leq C \left(\Vert \mathcal{U}_0 \Vert_{\mathcal{H}}  + \Vert (G_1,G_2) \Vert_{ (L^1_{T} L^{2})^2} \right) .$$
	\end{proof}
	\vspace{0.5cm}
	Now, we use Theorem \ref{th2.3} to deduce the global existence, the uniqueness and the Strichartz estimate for the solution of the semi-linear coupled wave system \eqref{system}. The proof of this result proceeds in two  steps:  
	\subsection{Local existence and uniqueness} Note that assumptions \eqref{nonlinearity} remain true if we replace $p$ by another larger exponent as long as it remains smaller than $5$. Therefore, we can assume without loss of generality that $3<p<5$.
	
	The aim of this paragraph is to prove the following result:
	\begin{theorem}\label{theo}
		Let	$R_{0}>0$. Then, there exists $T>0$ so that for any $\mathcal{U} _{0} \in \mathcal{H} $, with $\Vert \mathcal{U} _{0}\Vert_{\mathcal{H}}\leq R_{0}$, 
		system \eqref{2} has a unique solution 
		$$\mathcal{U} \in \mathcal{C}([0,T],\mathcal{H}) \cap \left(  \left( L^{q} _T L^{r} \right) ^2 \times \mathcal{C} \left( [0,T],L ^{2}( \Omega) \right) ^{2} \right) \; \text{ for } (q,r)= \left( \dfrac{2p}{p-3},2p \right).$$ 
		Moreover, there exists $C>0$ so that for any $\mathcal{U}_{0},\mathcal{V}_{0} \in \mathcal{H} $ with $ \Vert \mathcal{U}_{0} \Vert _{\mathcal{H}}+ \Vert\mathcal{V}_{0} \Vert _{\mathcal{H}}\leq R_{0}$, and $\mathcal{U}$, $\mathcal{V}$ the associated solutions, then we have $\Vert \mathcal{U}-\mathcal{V} \Vert _{\mathcal{C}([0,T],\mathcal{H})}\leq C  \Vert \mathcal{U}_{0} -\mathcal{V}_{0} \Vert _{\mathcal{H}}$.
		
	\end{theorem}
	
	The proof of this result is based on the fixed point theorem. First, we search a suitable Banach space. Second, thanks to Duhamel's formula, we search a fixed point for the map $\phi $ defined by:
	\begin{eqnarray}\label{phideu}
	\phi (\mathcal{U})(t)& = & S(t) \mathcal{U}_{0} + \int _{0}^{t} \mathcal{S}( t-s) \mathcal{F}(\mathcal{U})(s) ds. 
	\end{eqnarray}
	All the difficulty is to find a complete space on which $\phi$ is defined and strictly contracting. The choice of the resolution space will be guided by the Strichartz estimates which give a gain in integrability compared to the Sobolev embeddings.

	Let $T>0$, and for the admissible couple 
	\begin{eqnarray}\label{qr}
	(q,r)= \left( \dfrac{2p}{p-3},2p \right) ,
	\end{eqnarray}
	we define the following Banach space:
	\begin{eqnarray}
	E_{T} =  \mathcal{C}( [0,T],H_{0} ^{1}( \Omega) ) \cap L^{q}([0,T],L^{r} (\Omega)),\nonumber
	\end{eqnarray}
	with norm: 
	\begin{eqnarray}
	\Vert . \Vert _{E_{T} }  &: = & \max \left( \Vert . \Vert _{L^{q} _{T}L^{r}} , \Vert . \Vert _{L^{\infty} _{T}H^{1} _{0}} \right) ,\nonumber
	\end{eqnarray}
	then we introduce space $X_{T} = E_{T} ^{2} \times  \mathcal{C} ( [0,T],L ^{2}( \Omega) )^{2}$, with norm
	\begin{eqnarray}
	\Vert Y \Vert _{X_{T} }& : = & \Vert Y_{1} \Vert _{E_{T}^{2} } + \Vert Y_{2} \Vert _{(L^{\infty} _{T}L^{2} )^{2} }; \hspace{1,5cm} \text{for} \;  Y=(Y_{1}, Y_{2}) \; \in \; X_{T}\; .\nonumber  
	\end{eqnarray}

	We need the following lemmas:
	
	\begin{lemma}\label{lem2}
		Let $ \mathcal{U}_{0}  \in \mathcal{H} $. Then, for all $\mathcal{U} \in X_T $ we have
		\begin{eqnarray}\label{phiU}
		\Vert \phi (\mathcal{U}) \Vert_{ X_{T}} \lesssim \Vert \mathcal{U}_0 \Vert_{ \mathcal{H}} + \Vert \mathcal{F} (\mathcal{U}) \Vert _{\left( L_T ^1 L^2 \right) ^4} \; .
		\end{eqnarray}
		Moreover, for $\mathcal{U}, \mathcal{V} \in X_T $ we have
		\begin{eqnarray}\label{phiUV}
		\Vert \phi (\mathcal{U})- \phi (\mathcal{V}) \Vert_{ X_{T}} \lesssim  \Vert \mathcal{F} (\mathcal{U})-\mathcal{F} (\mathcal{V}) \Vert _{\left( L_T ^1 L^2 \right) ^4} \; .
		\end{eqnarray}
		
	\end{lemma}
	
	\begin{proof}
		Let $ \mathcal{U}_{0} =(u_0, v_0, u_1,v_1) \in \mathcal{H} $ and let $ \mathcal{U} \in X_T $. For all $t\geq 0$ we have from \eqref{phideu}
		\begin{eqnarray}\label{decompphiU} \phi (\mathcal{U})(t)= \mathcal{S}(t)\mathcal{U}_0 + \mathcal{W} (t), 
		\end{eqnarray}
		where $\mathcal{W} (t) = \displaystyle\int_{0}^{t} \mathcal{S}(t-s) \mathcal{F}(\mathcal{U})(s) ds $, then we have
		\begin{eqnarray}\label{PPhiU}
		\Vert  \phi (\mathcal{U}) \Vert_{X_T } \leq \Vert \mathcal{S}(.)\mathcal{U}_0 \Vert_{X_T } +  \Vert \mathcal{W} \Vert_{X_T }.
		\end{eqnarray}
		For all $ t \in \left[ 0,T \right]  $, we note $\psi(t)= \mathcal{S}(t)\mathcal{U}_{0}=(\varphi, \varphi_1 , \partial _t \varphi  ,\partial _t \varphi_1 )$.	Since $ (\mathcal{S}(t))_{t\geq 0} $ is a semi group of contractions, then we have
		$$ \forall t\geq0 , \; \Vert \mathcal{S}(t) \mathcal{U} _0  \Vert _{\mathcal{H}} \; \leq \; \Vert   \mathcal{U} _0  \Vert _{\mathcal{H}}, $$
		so we get 
		\begin{eqnarray}\label{est1}
		\Vert \psi \Vert _{L_T ^{\infty}(\mathcal{H})}= \Vert \mathcal{S}(.) \mathcal{U} _0  \Vert _{L_T ^{\infty}(\mathcal{H})} \; \leq\;   \Vert   \mathcal{U} _0  \Vert _{\mathcal{H}} . 
		\end{eqnarray}
		Furthermore, using the Strichartz estimate for the damped wave equation given in \\  \cite[Theorem 2.1]{2}, we get:
		\begin{eqnarray}
		\begin{cases}
		\Vert \varphi \Vert _{ L_{T}^{q} L^{r}} \lesssim  \Vert u_{0} \Vert _{ H^{1} } + \Vert u_{1} \Vert _{ L^{2} } + \Vert b \partial _t \varphi_1 \Vert _{ L_{T}^{1} L^{2}}, \\
		\Vert \varphi_1 \Vert _{ L_{T}^{q} L^{r}} \lesssim  \Vert v_{0} \Vert _{ H^{1} } + \Vert v_{1} \Vert _{ L^{2} } + \Vert b \partial _t \varphi \Vert _{ L_{T}^{1} L^{2}} ,
		\end{cases}\nonumber
		\end{eqnarray}
		then
		\begin{eqnarray}
		\Vert \psi \Vert _{ (L_{T}^{q} L^{r})^2 \times (L_{T}^{\infty} L^{2})^2 } &=& \Vert \varphi \Vert _{ L_{T}^{q} L^{r}} + \Vert \varphi_1 \Vert _{ L_{T}^{q} L^{r}} + \Vert \partial _t \varphi\Vert _{ L_{T}^{\infty} L^{2}} + \Vert \partial _t \varphi_1 \Vert _{ L_{T}^{\infty} L^{2}}, \nonumber \\
		& \lesssim &   \Vert \mathcal{U}_0 \Vert _{\mathcal{H}}  + \Vert \partial _t \varphi \Vert _{ L_{T}^{\infty} L^{2}} + \Vert \partial _t \varphi_1 \Vert _{ L_{T}^{\infty} L^{2}}  + \Vert b \partial _t \varphi \Vert _{ L_{T}^{1} L^{2}} + \Vert b \partial _t \varphi_1 \Vert _{ L_{T}^{1} L^{2}} . \nonumber 
		\end{eqnarray}
		Using the fact that $b \in L ^ \infty (\Omega) $, we obtain
		\begin{eqnarray}
		\begin{cases}
		\Vert b \partial _t \varphi \Vert _{ L_{T}^{1} L^{2}} \lesssim  T \Vert \partial _t \varphi \Vert _{ L_{T}^{\infty} L^{2}} ,  \\
		\Vert b \partial _t \varphi_1 \Vert _{ L_{T}^{1} L^{2}} \lesssim   T \Vert \partial _t \varphi_1 \Vert _{ L_{T}^{\infty} L^{2}} ,
		\end{cases} \nonumber
		\end{eqnarray}
		thus
		\begin{eqnarray}\label{est2}
		\Vert \psi \Vert _{ (L_{T}^{q} L^{r})^2 \times (L_{T}^{\infty} L^{2})^2 } \lesssim \Vert \mathcal{U}_0 \Vert _{\mathcal{H}} + \Vert \psi \Vert _{L_T ^{\infty}(\mathcal{H})}.  
		\end{eqnarray}
		Finally, from \eqref{est1} and \eqref{est2} we obtain
		\begin{eqnarray}\label{St}
		\Vert \mathcal{S}(.)\mathcal{U}_0 \Vert_{X_T } = \Vert \psi \Vert _{ X_T } \lesssim  \Vert \mathcal{U}_0 \Vert _{\mathcal{H}}.   
		\end{eqnarray}

		Recalling the definition of $\mathcal{W}$ in \eqref{decompphiU}, we denote $\mathcal{W} = (w, w_1 , \partial _t w  ,\partial _t w_1 )$.	On the other hand, due to Duhamel's formula,  $ (w, w_1) $ is the solution of problem
		\begin{eqnarray}
		\begin{cases}
		\partial _{tt} w  - \Delta w + a \partial _t w + b \partial _t w _ 1  + f_1(u)=0 & in \; \Omega \times \mathbb{R}_{+} ^{*} ,   \\
		\partial _{tt}	w _ 1  - \Delta w_ 1 - b \partial _t w   +f_2(v)=0  & in \; \Omega \times \mathbb{R}_{+} ^{*} ,  \\
		w = w_1 =0  & on \; \Gamma \times \mathbb{R}_{+} ^{*} , \\
		(w, w_1 , \partial _t w  ,\partial _t w_1 ) (0) =0  & in \; \Omega  .  
		\end{cases}\nonumber
		\end{eqnarray}
		We denote $\mathcal{U}=(u,v,\tilde{u},\tilde{v} )\in X_{T}$, where $(u,v)$ are the first components of $\mathcal{U}$.
		Then, we have 
		$$ \Vert \mathcal{W} \Vert_{X_T } = \Vert (w , w_1) \Vert _{(L_T ^{\infty}H_0 ^1)^2 } + \Vert (w, w_1) \Vert _{(L_T ^{q}L^r)^2 } + \Vert (\partial _t w  , \partial _t w_1 ) \Vert _{(L_T ^{\infty}L ^2)^2 }. $$
		Using the Strichartz estimate given in Theorem \ref{th2.3}, we get
		
		\begin{eqnarray}
		\Vert (w , w_1) \Vert _{(L_T ^{q}L^r)^2 } &\lesssim &\Vert \mathcal{F}(\mathcal{U}) \Vert _{(L_T ^{1}L^2) ^4 } ,\nonumber
		\end{eqnarray}
		
		then we obtain
		\begin{eqnarray}\label{est4}
		\Vert \mathcal{W} \Vert_{X_T } &\lesssim & \Vert (w , w_1) \Vert _{(L_T ^{\infty}H_0 ^1)^2 } + \Vert (\partial _t w  , \partial _t w_1 ) \Vert _{(L_T ^{\infty}L ^2)^2 }+ \Vert \mathcal{F}(\mathcal{U}) \Vert _{(L_T ^{1}L^2) ^4 }, \nonumber \\ 
		& \lesssim & \Vert \mathcal{W} \Vert _{L_T ^{\infty}(\mathcal{H}) }  + \Vert \mathcal{F}(\mathcal{U}) \Vert _{(L_T ^{1}L^2) ^4 }. 
		\end{eqnarray}
		From Lemma \ref{Energies identities}, we have for all $t\geq 0 $
		\begin{eqnarray}
		\Vert \mathcal{W}(t,.) \Vert _{\mathcal{H}} \lesssim  \Vert  \mathcal{F} (\mathcal{U}) \Vert _{(L^1_TL^2)^4},\nonumber
		\end{eqnarray}
		so
		\begin{eqnarray}\label{est3}
		\Vert \mathcal{W} \Vert _{L_T ^{\infty}(\mathcal{H}) }  \lesssim  \Vert  \mathcal{F} (\mathcal{U}) \Vert _{(L^1_TL^2)^4}.
		\end{eqnarray}
		Consequently, injecting \eqref{est3} in \eqref{est4}  we get
		\begin{eqnarray}\label{Fu}
		\Vert \mathcal{W} \Vert_{X_T } \lesssim \Vert \mathcal{F}(\mathcal{U}) \Vert _{(L_T ^{1}L^2) ^4 }.
		\end{eqnarray}
		Substituting \eqref{St} and \eqref{Fu} in \eqref{PPhiU}, we get the estimate \eqref{phiU}. Indeed, \eqref{phiUV} is obtained similarly.
	\end{proof}
	\begin{lemma}\label{lem1}
		Let $\mathcal{U} \in X_T $, then we have the following estimate:
		\begin{eqnarray}\label{lemm}
		\Vert \mathcal{F}(\mathcal{U}) \Vert _{\left( L_T ^1 L^2 \right) ^4} \lesssim \Vert \mathcal{U} \Vert _{X_T} (T+T^\theta \Vert \mathcal{U} \Vert _{X_T} ^{p-1}) ,
		\end{eqnarray}
		where $\theta = \dfrac{5-p}{2} > 0$. Similar to the proof of \eqref{lemm} we can show that for every\\ $\; \mathcal{U}, \; \mathcal{V} \in X_T $ we have 
		\begin{eqnarray}\label{lemmm}
		\Vert \mathcal{F}(\mathcal{U}) -  \mathcal{F}(\mathcal{V}) \Vert _{\left( L_T ^1 L^2 \right) ^4} \lesssim \Vert \mathcal{U} - \mathcal{V} \Vert _{X_T} \left( T+T^\theta (\Vert \mathcal{U} \Vert _{X_T} ^{p-1} + \Vert \mathcal{V} \Vert _{X_T} ^{p-1})\right)  .
		\end{eqnarray}
	\end{lemma}
	\begin{proof}
		Let $\mathcal{U}=(u,v,u_1,v_1) \in X_T$. We have
		\begin{eqnarray}\label{FU}
		\Vert \mathcal{F}(\mathcal{U}) \Vert _{\left( L_T ^1 L^2 \right) ^4} = \Vert f_1(u)  \Vert _{ L_T ^1 L^2 } + \Vert f_2(v) \Vert _{ L_T ^1 L^2 }. 
		\end{eqnarray}
		Using the fact that $f_i(0)=0$, for $i=1,2$, we have for all $s \in \mathbb{R} $,
		\begin{eqnarray}\label{f(s)}
		\vert	f_i(s) \vert  =  \vert f_i(s) - f_i(0) \vert  & \lesssim & \vert	s \vert  \underset{ t \in \left[  0,1 \right] }{ \sup } \vert f_i^\prime (ts) \vert 
		\lesssim \; \vert	s \vert  \underset{ t \in \left[  0,1 \right] }{ \sup } \left( 1 + \vert ts \vert \right) ^{p-1} , \nonumber \\
		& \lesssim & \vert	s \vert  \left( 1 + \vert s \vert \right) ^{p-1} 
		\lesssim  \vert	s \vert  \left( 1 + \vert s \vert ^{p-1}  \right) . 
		\end{eqnarray}
		Then, we get 
		\begin{eqnarray}\label{f(u)}
		\Vert f_1(u) \Vert _{ L_T ^1 L^2 } &\lesssim & \Vert u (1 + \vert u \vert ^{p-1} ) \Vert _{ L_T ^1 L^2 }\leq  \Vert u  \Vert _{ L_T ^1 L^2 } + \Vert u \vert u \vert ^{p-1}  \Vert _{ L_T ^1 L^2 } , \nonumber \\
		&\lesssim & T \Vert u  \Vert _{ L_T ^\infty L^2 } + \Vert  u \vert u \vert ^{p-1}  \Vert _{ L_T ^1 L^2 } .
		\end{eqnarray}
		We have for all $t \in \left[ 0, T \right] $
		\begin{eqnarray}
		\Vert u  \vert u \vert ^{p-1} \Vert _{  L_T ^1 L^2}  &= &  \int_{0} ^{T} \Vert u \vert u \vert ^{p-1} \Vert _{ L^2}  ds  			 =   \int_{0} ^{T} \Vert u  \Vert _{ L^{2p} }^{p} ds  
		=     \Vert u  \Vert _{ L_{T}^pL^{2p} } ^{p} . \nonumber
		\end{eqnarray}
		Since $3<p <5 $, we use the fact that $  \frac{p-3}{2} + \frac{5-p}{2} =1 $; and applying the H\"{o}lder inequality for the couple $\left( \frac{2}{p-3}, \frac{2}{5-p} \right) $, we get
		\begin{eqnarray}\label{2.24}
		\Vert u  \Vert _{L_T ^{p} L^{2p}}  & = & \left(  \int_{0}^{T} 1 \; . \Vert u (t,.) \Vert _{L^{2p}} ^{p} dt \right) ^{\frac{1}{p}}, \nonumber \\
		&\lesssim &  \left( \int_{0}^{T} 1^{\frac{2}{5-p}} dt \right) ^{\frac{5-p}{2p}}  \left( \int_{0}^{T} \Vert u (t,.) \Vert _{L^{2p}} ^{\frac{2p}{p-3}} dt \right) ^{\frac{p-3}{2p}}  , \nonumber \\
		&\lesssim & T ^{\frac{\theta}{p}} \Vert u  \Vert _{L_T ^{\frac{2p}{p-3}} L^{2p}},
		\end{eqnarray}
		where $ \theta = \frac{5-p}{2}>0$. Then, from \eqref{f(u)} we obtain
		\begin{eqnarray}\label{fu}
		\Vert f_1(u) \Vert _{ L_T ^1 L^2 } &\lesssim &  T \Vert u  \Vert _{ L_T ^\infty L^2 } + T ^{\theta} \Vert u  \Vert _{L_T ^{\frac{2p}{p-3}} L^{2p}} ^p,\nonumber \\
		&\lesssim &  T \Vert \mathcal{U}   \Vert _{X_T} + T ^{\theta} \Vert \mathcal{U}   \Vert _{X_T} ^p,\nonumber \\
		&\lesssim &   \Vert \mathcal{U} \Vert _{X_T}\left(  T+ T ^{\theta} \Vert \mathcal{U}   \Vert _{X_T} ^{p-1}\right)  .
		\end{eqnarray}
		In addition, 
		\begin{eqnarray}\label{fv}
		\Vert f_2(v) \Vert _{ L_T ^1 L^2 } &\lesssim &   \Vert \mathcal{U}  \Vert _{X_T}\left(  T+ T ^{\theta} \Vert \mathcal{U}  \Vert _{X_T} ^{p-1}\right)  .
		\end{eqnarray}
		Finally, injecting \eqref{fu} and \eqref{fv} in \eqref{FU}, we get estimate \eqref{lemm}.

		Let	$ \mathcal{U}=(u,v,u_1,v_1),\; \mathcal{V} =(w, y,w_1 ,y_1) $ in $X_T$. We have
		\begin{eqnarray}\label{FUV}
		\Vert \mathcal{F}(\mathcal{U}) -  \mathcal{F}(\mathcal{V}) \Vert _{\left( L_T ^1 L^2 \right) ^4} &=& \left\Vert f_1(w) - f_1(u) \right\Vert _{L_{T}^{1}L^{2}} + \left \Vert f_2(y) - f_2(v) \right \Vert _{L_{T}^{1} L^{2}}.  
		\end{eqnarray}
		For $i=1,2$, we have for all $ s_1, \; s_2 \in \mathbb{R} $ 
		\begin{eqnarray}
		\vert f_i (s_1)-f_i(s_2) \vert & \lesssim & \int_{0}^{1} \vert f_i ^\prime (t s_1 + (1-t) s_2) \vert \vert s_1-s_2 \vert dt ,\nonumber \\
		& \lesssim & \int_{0}^{1} \left(  1+ \vert  t s_1 + (1-t) s_2) \vert \right) ^{p-1} \vert s_1-s_2 \vert dt ,\nonumber \\
		& \lesssim & \vert s_1-s_2 \vert \left( 1 + \vert s_1 \vert ^{p-1} + \vert s_2 \vert ^{p-1} \right).\nonumber
		\end{eqnarray}
		Then,  using similar computation as before and H\"older estimates, we get 
		\begin{eqnarray}\label{yv}
		\left\Vert f_1(y) - f_1(v) \right\Vert _{L_{T}^{1}L^{2}} & \lesssim &  \Vert y-v \Vert _{ L_T ^1 L^2 } +T ^{\theta} \Vert y-v \Vert _{ L_T ^{\frac{2p}{p-3}} L^{2p}}  \left(\Vert y  \Vert _{L_T ^{\frac{2p}{p-3}} L^{2p}} ^{p-1} +\Vert v  \Vert _{L_T ^{\frac{2p}{p-3}} L^{2p}} ^{p-1}\right),\nonumber \\
		& \lesssim &  \left\Vert \mathcal{U}-\mathcal{V} \right\Vert _{X_{T}}  \left( T + T ^{\theta}(\left\Vert \mathcal{U} \right\Vert _{X_{T}} ^{p-1}+\left\Vert \mathcal{V} \right\Vert _{X_{T}} ^{p-1}) \right) .
		\end{eqnarray}
		Furthermore,
		\begin{eqnarray}\label{wu}
		\left\Vert f_2(w) - f_2(u) \right\Vert _{L_{T}^{1}L^{2}} & \lesssim &  \left\Vert \mathcal{U}-\mathcal{V} \right\Vert _{X_{T}}  \left( T + T ^{\theta}(\left\Vert \mathcal{U} \right\Vert _{X_{T}} ^{p-1}+\left\Vert \mathcal{V} \right\Vert _{X_{T}} ^{p-1}) \right) .
		\end{eqnarray}
		Finally, injecting \eqref{yv} and \eqref{wu} in \eqref{FUV}, we get estimate \eqref{lemmm}.
	\end{proof}
	\begin{proof}[Proof of Theorem \ref{theo}]
		Let $ \mathcal{U}_0 \in \mathcal{H}$, $T> 0 $ and $ \mathcal{U}=(u,v,u_1,v_1),\; \mathcal{V} =(w, y,w_1 ,y_1) $ in $X_T$. We assume $\mathcal{U}$ and $\mathcal{V}$ in $B_{X_{T}}(0,R)$ with $R$ to be chosen large.	
		
		We have for all $t \in \left[ 0,T \right] $ 
		$$ \phi ( \mathcal{U}) (t)  - \phi ( \mathcal{V}  )(t) = \int _{0} ^{t} \mathcal{S}(t-s) \left(  \mathcal{F}(\mathcal{U}) - \mathcal{F}(\mathcal{V} ) \right) (s) ds. $$
		Using $\eqref{phiUV}$ then
		\begin{eqnarray}
		\Vert \phi ( \mathcal{U} ) - \phi (\mathcal{V} ) \Vert _{X_{T}} & \lesssim & \Vert \mathcal{F}(\mathcal{U})-\mathcal{F}(\mathcal{V} ) \Vert _{( L_{T}^{1} L^{2})^4}. \nonumber 
		\end{eqnarray}
		Now, using Lemma \ref{lem1}, we conclude that there exists $C> 0$ such that 
		\begin{eqnarray}\label{2.30}
		\Vert \phi ( \mathcal{U}  ) - \phi ( \mathcal{V}  ) \Vert _{X_{T}} & \leq &  C  \Vert \mathcal{V}  - \mathcal{U} \Vert_{X_{T}} \left( T+T^\theta ( \Vert \mathcal{U}  \Vert _{ X_{T}} ^{p-1} + \Vert \mathcal{V}  \Vert _{ X_{T}} ^{p-1} ) \right)  , \nonumber \\
		\label{ineqPhidiff}& \leq &   C \Vert \mathcal{V}  - \mathcal{U} \Vert_{X_{T}} \left( T+T^\theta 2 R ^{p-1 }  \right)   . 
		\end{eqnarray}
		By Lemma \ref{lem2} and \eqref{lemm}, we obtain 
		\begin{eqnarray}
		\Vert \phi ( \mathcal{U}  ) \Vert _{X_{T}} & \leq &  C \Vert \mathcal{U}_0 \Vert_{\mathcal{H}} + CR\left( T+T^\theta 2 R ^{p-1 }  \right)   . 
		\end{eqnarray}
		We choose $R>2C  \Vert \mathcal{U}_0 \Vert_{\mathcal{H}} $ and $T$ such that 
		\begin{eqnarray}\label{temps}
		T+ 2 T^\theta R^{p-1}  < \frac{1}{2C}\; .
		\end{eqnarray}

		We deduce that $\phi: B_{X_{T}}(0,R) \subset X_T \longrightarrow B_{X_{T}}(0,R)$ and $\phi $ is a contraction. Then by the fixed point theorem there is a unique $\mathcal{U} \in B_{X_{T}}(0,R) \subset X_T$ such that $ \phi (\mathcal{U}) = \mathcal{U} $, where $ \mathcal{U}$ is a solution of system \eqref{system}, Which completes the proof of Theorem \ref{theo}.
		\subsubsection*{\textbf{Uniqueness}}
		We prove now the uniqueness of the solution. Let $\mathcal{U}$ and $\mathcal{V}$ be two solutions of \eqref{2} in $X_T$.  We note $M = \max \left( \Vert \mathcal{U} \Vert _{X_T} , \Vert \mathcal{V} \Vert _{X_T} \right)$. \\
		Using the Duhamel formula, we have $\mathcal{U} = \phi (\mathcal{U})$ et $\mathcal{V} = \phi (\mathcal{V})$. Moreover 
		for all $ T_{0} \in \left[ 0, T \right] $, using \eqref{2.30} we have  
		\begin{eqnarray}
		\Vert \mathcal{U} - \mathcal{V} \Vert _{X_T} &= & \Vert \phi (\mathcal{U}) - \phi (\mathcal{V}) \Vert _{X_T},\nonumber \\
		& \leq   & C \Vert \mathcal{U} - \mathcal{V} \Vert _{X_T} \left( T_0 + T_{0}^\theta \left( \Vert \mathcal{U}  \Vert _{X_T}^{p-1} + \Vert \mathcal{V}  \Vert _{X_T}^{p-1} \right) \right) ,\nonumber \\
		& \leq & C \Vert \mathcal{U} - \mathcal{V} \Vert _{X_T} \left(  T_0+ T_{0}^\theta 2 M^{p-1} \right) .\nonumber
		\end{eqnarray} 
		If $ T_0 \leq T $ and $T_0 + T_{0}^\theta 2 M^{p-1} \leq \dfrac{1}{2C} $, then $ \Vert \mathcal{U} - \mathcal{V} \Vert _{X_T} \leq \frac{1}{2} \Vert \mathcal{U} - \mathcal{V} \Vert _{X_T} $. Thus, $\mathcal{U}= \mathcal{V}$ on $ \left[0 , T_{0} \right] .$ 
		We iterate the same process on $ \left[T_{0} ,2 T_{0} \right]$, and we obtain $\mathcal{U}= \mathcal{V}$ on $ \left[T_{0} ,2 T_{0}  \right] .$ \\
		By iteration we can show that $\mathcal{U}= \mathcal{V}$ on $ \left[0 , T  \right] .$ Finally, we deduce that  system \eqref{2} has a unique solution in $X_T$.
		\subsubsection*{\textbf{Strichartz estimate}}
		Let $ \mathcal{U}=(u,v,u^\prime , v^\prime) $ be the solution of system \eqref{2}. Then lemma \ref{lem1} assures that $ \mathcal{F}(\mathcal{U}) \in L^1 _T L^2 $. Furthermore, using the classical energy estimate \eqref{equiv} and \eqref{lemm}, we obtain 
		\begin{eqnarray}
		\Vert \mathcal{F}(\mathcal{U}) \Vert _{L_T ^1 L^2} & \lesssim &  T  \Vert \mathcal{U} \Vert _{X_T }  +  T ^\theta \Vert \mathcal{U} \Vert _{X_T } ^p , \nonumber \\
		& \lesssim & T  E_{u,v}(t)  +   T ^\theta E_{u,v}(t) ^p , \nonumber \\
		& \lesssim & T  E_0 + T ^\theta E_0 ^p , \nonumber \\
		& \leq & C ( T , E_0 ) . \nonumber
		\end{eqnarray} 
		Therefore, since we know now that equation \eqref{2} is satisfied, we get from Theorem \ref{th2.3} that for all $(q,r)$ satisfying \eqref{coupleadm} we have
		\begin{eqnarray}
		\Vert (u,v) \Vert _{ (L^q(\left[ 0,T \right] ,L^r(\Omega)))^2} &\leq & C \left(\Vert \mathcal{U}_0 \Vert_{\mathcal{H}}  + \Vert \mathcal{F}(\mathcal{U}) \Vert _{L_T ^1 L^2} \right), \nonumber \\
		&\leq &C ( T , E_0, \Vert \mathcal{U} _0\Vert _{\mathcal{H} } ) .
		\end{eqnarray}
		
		\subsubsection*{\textbf{Stability estimates}}
		To get the stability, we write again the Duhamel formula  for the difference of the solutions to get as in \eqref{2.30}
		\begin{eqnarray}
		\Vert  \mathcal{U}   - \mathcal{V}   \Vert _{X_{T}} & \leq &C \Vert \mathcal{U}_{0} -\mathcal{V}_{0} \Vert _{\mathcal{H}}+  C  \Vert \mathcal{V}  - \mathcal{U} \Vert_{X_{T}} \left( T+T^\theta ( \Vert \mathcal{U}  \Vert _{ X_{T}} ^{p-1} + \Vert \mathcal{V}  \Vert _{ X_{T}} ^{p-1} ) \right)  , \nonumber \\
		\label{ineqPhidiff}& \leq & C \Vert \mathcal{U}_{0} -\mathcal{V}_{0} \Vert _{\mathcal{H}}+  C \Vert \mathcal{V}  - \mathcal{U} \Vert_{X_{T}} \left( T+T^\theta 2 R ^{p-1 }  \right)   . 
		\end{eqnarray}
		It gives the expected estimate when $T$ is small enough.

	\end{proof}
	\subsection{Global existence}
	The proof of the global existence is mainly of the type:\\ " energy decay " + "explosion criterion" implies "global existence". The local wellposedness theory obtained in Theorem \ref{theo} classically implies the existence of a unique maximal solution with a prescribed initial condition. Hence, to get a global solution, we need the following lemma 
	\begin{lemma}{(explosion criterion)}
		\label{lmexplosioncrit}
		Let $\mathcal{U}_0 \in \mathcal{H}$.	We denote by $T^*$ the supremum of all $T > 0$ for which there exists a solution $ \mathcal{U} \in \mathcal{C}(\left[ 0,T^* \right[, \mathcal{H} )$	of \eqref{2}.
		Then we have:
		\begin{eqnarray}
		T^*< + \infty \Longrightarrow \displaystyle \lim_{t \to T^*} \Vert \mathcal{U}(t)\Vert _{\mathcal{H}} = + \infty .
		\end{eqnarray}
	\end{lemma}
	
	\begin{proof}
		Let $ \mathcal{U} \in \mathcal{C}(\left[ 0,T^* \right[, \mathcal{H} )$ be a maximal solution of system \eqref{2} such that $ T^*< \infty $. We argue by contradiction and assume that $\Vert \mathcal{U}(t)\Vert _{\mathcal{H}} $ does not converge to $+\infty$ as $t\to T^{*}$. In particular, there exists $M>0$ and a sequence of time $\left( t_n \right)_{n \in \mathbb{N}} \subset \left[ 0,T^* \right[$, converging to $ T^* $ such that for all $n \in \mathbb{N}$,  we have 
		\begin{eqnarray}\label{exp}
		\Vert \mathcal{U}(t_n) \Vert _{\mathcal{H}} \leq M .
		\end{eqnarray}
		Let $ T_{0}=\min(D,D^{1/\theta})$ with $D=\frac{1}{C_1(1+2M^{p-1})}$, $C_1 > C $ and $C$ be the constant given in \eqref{2.30}, which is the condition that ensures the well posed in time $T$ in Theorem \ref{theo}. We choose $t_N$, for some $N$ in $\mathbb{N}$, such that $t_N +T_0 > T^* $. \\
		
		We will study the following problem :
		\begin{eqnarray}\label{newsolution}
		\begin{cases}
		\partial _t	\mathcal{V} (t)= \mathcal{A} \mathcal{V} (t) + \mathcal{F} (\mathcal{V}), \\
		\mathcal{V}(0)=\mathcal{U}(t_N),
		\end{cases}
		\end{eqnarray}
		where $ \mathcal{V}(t) = \mathcal{U}(t +t_N) $. 
		
		With the previous choice of $T_{0}$ and $M$, Theorem \ref{theo} ( precisely (2.34) in the proof of Theorem 2.4) allows to build a solution of system \eqref{newsolution} on a time interval of length $ T_{0}$ because 
		$$ T_0 + 2 T_0 ^ \theta \Vert \mathcal{U}(t_{0} )\Vert _{\mathcal{H}} ^{p-1} \leq D+2 D\Vert \mathcal{U}(t_{0} ) \Vert _{\mathcal{H}} ^{p-1} = \dfrac{1}{C_1}  < \dfrac{1}{C}. $$
		
		We consider the function $\mathcal{W}$ defined by 
		$$ \mathcal{W} (t)= \begin{cases}
		\mathcal{U} (t) \quad \text{on} \left[ 0, t_{N} \right[ , \\
		\mathcal{V}(t-t_{N}) \quad \text{on} \left[ t_N, t_N +t_0\right[ .
		\end{cases}
		$$
		$\mathcal{W}$ is well defined and indeed constitutes a solution of system \eqref{2} on $ \left[ 0, t_N +t_0\right[ $, which contradicts the maximality of $ \mathcal{U} $. Then \eqref{exp} is false and we have 
		$$ \Vert \mathcal{U}(t_n) \Vert _{\mathcal{H}} \displaystyle \xrightarrow [n\to +\infty]{} +\infty .$$
	\end{proof}
	
	\begin{corollary}
		Let $\mathcal{U} _{0} \in \mathcal{H} $. Then, the system \eqref{2} has a unique solution $\mathcal{U}$ on $[0,+\infty)$. Moreover, for any  $T>0$, we have 
		$$\mathcal{U} \in \mathcal{C}([0,T],\mathcal{H}) \cap \left(  \left( L^{q} _T L^{r} \right) ^2 \times \mathcal{C} \left( [0,T],L ^{2}( \Omega) \right) ^{2} \right) \; \text{ for } (q,r)= \left( \dfrac{2p}{p-3},2p \right).$$ 
	\end{corollary}
	\begin{proof}
		As before, let $T^{*}\in (0,+\infty]$ be the maximal time existence. Using the fact that the energy is a decreasing function in time and estimate \eqref{equiv}, for all $t\in [0,T^{*})$  we get $\Vert \mathcal{U}(t) \Vert _{\mathcal{H}}\leq C E(\mathcal{U}(t))\leq E(\mathcal{U}_{0}) $. In particular, the explosion criteria of Lemma \ref{lmexplosioncrit} gives $T^* = + \infty $ and then the unique solution given by Theorem \ref{theo} is a global solution.
	\end{proof}
	
	\section{Stabilization}
	\label{s:stab}
	This section is addressed to prove the uniform stability given in Theorem \ref{th1}. Due to \cite[Proposition 2.5]{2}, this proof follows from the well-known criterion for exponential decay. Our estimate is the following.
	
	\begin{theorem} \label{th3.6}
		Let	$f_i \in \mathcal{C}^1(\mathbb{R},\mathbb{R})$, for $i=1,2$, satisfy \eqref{nonlinearity} and let $E_0 > 0$. We assume that $f_i$ is analytic, $ \omega_{b} $ satisfies the geometric control condition and $ supp(b) \subset \omega_{a} $, then there exist $T>0 $ and $C>0$ such that for   $ (u,v) $ solution of \eqref{system} where
		$ E_{u,v} (0) \leq E_{0} $, satisfies 
		\begin{eqnarray}\label{IO}
		E_{u,v} (0) \leq C \int _{0} ^{T} \int _{\Omega } a(x) \vert \partial _t u (x,t) \vert ^2 dx dt .
		\end{eqnarray}
		
	\end{theorem}
	In order to prove the previous result, we will strongly use the exponential decay for the linear semigroup.
	\subsection{Exponential decay of linear semigroup}
	In this paper, we will use the exponential decay for the linear semigroup $(\mathcal{S}(t))_{t\geq 0 } $ when the coupling region satisfies the geometric control condition and is contained in the damping region given in \cite{radhia2022local} (see also \cite{cui2020observability}).
	\begin{theorem}\label{th2}
		We assume that $ \omega_{b} $ satisfies the geometric control condition and \\ $ supp(b) \subset \omega_{a} $. Then, there exists $C>0$  so that for all initial data $\mathcal{U}_{0}  \in \mathcal{H} $ we have
		\begin{eqnarray}
		\Vert \mathcal{S}(t) \mathcal{U}_0 \Vert _{\mathcal{H}} \leq C e^{-\beta t} \Vert \mathcal{U}_0 \Vert_{\mathcal{H}}, 
		\end{eqnarray}
		for all $ t \geq 0 $, where $ \beta $ is a positive constant.
	\end{theorem}
	For $\epsilon \in \left[ 0, 1 \right] $, we note $ \mathcal{H} ^\epsilon = \mathcal{D}\left( \left( -\Delta_{D} \right) ^{\frac{1+\epsilon}{2}}\right) ^{2}\times \mathcal{D}\left( \left( -\Delta_{D}\right) ^{\frac{\epsilon}{2}}\right) ^{2} $ where $\Delta_{D}$ is the Dirichlet Laplacian. We notice that for $\epsilon \in \left[ 0, 1 \right]\setminus \{1/2\} $, which will always be the case in later use, we have $ \mathcal{H} ^\epsilon=\left( H^{1+ \epsilon }(\Omega) \cap H^1_0 (\Omega) \right) ^2 \times  \left(H^\epsilon _0 (\Omega) \right) ^2 $\footnote{We denote by $H^s_0$ the completion of $C^{\infty}_{c}(\Omega)$ for the norm of $H^{s}(\Omega)$. $H^s_0=H^{s}(\Omega)$ is the usual Sobolev space for $s\in [0,1/2)$ and $H^s_0= \left\lbrace u \in H^s , \;u =0 \text{ on } \Gamma \right\rbrace $ for $s\in (1/2,1]$.}, see for instance \cite{G:67}. Noticing that  $\mathcal{D}((-\Delta_{D})^{\theta})=[\mathcal{D}(-\Delta_{D}),L^{2}(\Omega)]_{\theta}$ for $\theta\in [0,1]$, that is are interpolation spaces.
	
	\begin{corollary}\label{cor1}
		We assume that $ \omega_{b} $ satisfies the geometric control condition and \\ $ supp(b) \subset \omega_{a} $. For all initial data $\mathcal{U}_{0}  \in \mathcal{H}_{\epsilon} $ we have
		\begin{eqnarray}
		\forall t \geq 0, \; \Vert \mathcal{S}(t) \mathcal{U}_0 \Vert _{\mathcal{H}^\epsilon} \leq C e^{-\beta t} \Vert \mathcal{U}_0 \Vert_{\mathcal{H}^\epsilon}, 
		\end{eqnarray}
		where $ \beta $ is a positive constant,  and $\epsilon \in \left[ 0, 1 \right]$.
	\end{corollary}
	The corollary can be proved for $\epsilon=1$ by noticing that $ \mathcal{H}^{1}= \mathcal{D}(\mathcal{A})$ with equivalent norm as long as $a$ and $b$ are in $L^{\infty}$ and the fact that $\mathcal{A}$ commutes with $\mathcal{S}(t)$. An interpolation argument allows us to conclude for $\epsilon\in[0,1]$. We refer for instance to \cite[Proposition 2.3]{2} for a similar proof in the scalar case. 
	
	\begin{remark}
		\label{rklienODE}
		Note that following the results of \cite{cui2020observability} in the case of a manifold without boundary, the weak observability (that is the observability up to a weaker norm, corresponding to high frequency, see Definition 1.2 in \cite{cui2020observability}) of our system is equivalent to the observability of the ODE systems 
		\begin{equation}\label{systemODEdamped}
		\begin{cases}
		\dot{x}(s) + \frac{1}{2}a(\gamma_{\rho_{0}}(s))x(s)  +\frac{1}{2} b(\gamma_{\rho_{0}}(s))y(s)&= 0,  \\
		\dot{y}(s)   - \frac{1}{2} b(\gamma_{\rho_{0}}(s))x(s)&= 0,  \\
		(x(0), y(0))&=(x_{0},y_{0}),
		\end{cases}
		\end{equation}
		where $s\mapsto \gamma_{\rho_{0}}(s)$ is the geodesic flow starting at point $\rho_{0}\in T^{*}M$. The observation is made by $\frac{1}{4}\int_{0}^{T} \left|a(\gamma_{\rho_{0}}(s))x(s)\right|^{2}ds$. It is well known that it is equivalent to the observation of the undamped system
		\begin{equation}\label{systemODEundamped}
		\begin{cases}
		\dot{x}(s) +\frac{1}{2} b(\gamma_{\rho_{0}}(s))y(s)&= 0 , \\
		\dot{y}(s)  - \frac{1}{2}b(\gamma_{\rho_{0}}(s))x(s)&= 0 , \\
		(x(0), y(0))&=(x_{0},y_{0}).
		\end{cases}
		\end{equation}
		With $X(s)=(x(s), y(s))$ and $b(s)=b(\gamma_{\rho_{0}}(s))$ for simplicity, it can be written \\ $\dot{X}(s)=- \dfrac{1}{2} b(s)MX(s)$ with $M=\begin{pmatrix}0&1\\-1&0\end{pmatrix}$. The solution is then $X(s)= e^{- B(s)M}X_{0}  $ with $B(s)=\dfrac{1}{2} \displaystyle \int_{0}^{s} b(\tau)d\tau$. 
		It is easy to solve and we can prove (see Section 5.2 of the arXiv preprint arXiv:1810.00512 which is an expanded version of \cite{cui2020observability})
		that the observability is equivalent to the following assumption.
		
		For any $\gamma_{0}\in T^{*}M$, there exists $0<t_1<t_2< T$, such that
		$$
		a( \gamma_{\rho_{0}}(t_{1}))\neq 0,a( \gamma_{\rho_{0}}(t_{2}))\neq 0, \int_{t_{1}}^{t_{2}}b( \gamma_{\rho_{0}}(\tau))d\tau\not\in 2 \pi \mathbb{Z} .
		$$
		Our assumption trivially implies this fact. Note that in the case of a domain with boundary, the geodesic flow should be replaced by the generalized broken bicharacteristic flow of Melrose-Sj\"ostrand, see \cite{ML:78}.
	\end{remark}
	
	\subsection{Exponential decay of semilinear coupled wave system}

	In this section, we need the following results. It is written this way in \cite[Corollary 4.2]{2}, but it is a straightforward corollary of \cite[Theorem 8]{dehman2003stabilization}.
	\begin{lemma}\label{L5}
		Let $R>0$ and $ T>0$. Let $s  \in  \left[   0, 1  \right[  $ and let $\epsilon = \min (1-s,(5-p)/2,(17-3p)/14) > 0$. There exist $(q,r)$ satisfying \eqref{coupleadm} and $C >0 $ such that the following property holds. If $v \in L^\infty (\left[ 0, T \right] , H ^{1+ s}(\Omega) \bigcap H_0^1 (\Omega))$ is a function with finite Strichartz norms $ \Vert v \Vert _{L^q _T L^r} \leq R $, then for $i=1,2$,  $f_i(v) \in L^1 (\left[ 0, T \right] , H_{0} ^{s+ \epsilon}(\Omega))$ and moreover 
		\begin{eqnarray}
		\Vert f_i(v)  \Vert _{L^1 _T H_{0} ^{s+ \epsilon}(\Omega)} \leq C \Vert v \Vert _{L^\infty _T  H ^{1+ s}(\Omega) \bigcap H_0^1 (\Omega) }.
		\end{eqnarray}
		The constant $C$ depends only on $\Omega $, $(q,r)$, $R$ and the constant in estimate \eqref{equiv}. 
	\end{lemma}
	\begin{proposition}(asymptotic compactness property \cite{2}) \item \label{L2}
		Let	$f_i \in \mathcal{C}^1(\mathbb{R},\mathbb{R})$, for $i=1,2$, satisfy \eqref{nonlinearity}, let $(\mathcal{U}_{n,0})_{n \geq 0 }$ be a sequence of initial data which is bounded in $\mathcal{H}$ and let $(\mathcal{U}_{n})_{n \geq 0 }$ be the corresponding solutions of the semilinear coupled wave system \eqref{system}. Let $ (t_n)  \in \mathbb{R}_+$ be a sequence of times such that $t_n \longrightarrow + \infty $ when $n$ goes to $+ \infty $.

		Then there exist subsequences $ (\mathcal{U}_{\phi (n)})$ and $ (t_{\phi (n)})$ and a global solution $ \mathcal{U}_{\infty}$ of \eqref{system} such that 
		$$   \mathcal{U}_{\phi (n)}(t_{\phi (n)}+.)\longrightarrow \mathcal{U}_{\infty} (.) \; \text{ in}\;\; \mathcal{C}^0 ( [0,T[, \mathcal{H}) \quad \text{for all } T >0.$$
	\end{proposition}
	
	\begin{proof}
		Due to the equivalence between the norm of $\mathcal{H}$ and the energy given in \eqref{equiv} and since the energy is decreasing in time, we know that $\mathcal{U}_n(t)$ is bounded in $\mathcal{H}$ uniformly with respect to $n$ and $t\geq 0$. In particular, the sequence $(\mathcal{U}_n(t_{n}))_{n\in \mathbb{N}}$ is a bounded sequence with value in the Hilbert space $\mathcal{H}$ from which we can extract a subsequence, so we can assume that it weakly converges to a limit called $\mathcal{U}_{\infty,0}\in \mathcal{H}$. Using the Duhamel formula, we have 
		\begin{eqnarray}
		\mathcal{U}_{n} (t_{n}) & = & \mathcal{S}(t_{n})\mathcal{U}_{n} (0) +\int_{0}^{t_{n}} \mathcal{S}(s) \mathcal{F}(\mathcal{U}_{n})(t_{n} -s) ds , \nonumber \\[0,4cm]
		& = & \mathcal{S}(t_{n})\mathcal{U}_{n}(0) +\int_{0}^{\lfloor t_{n} \rfloor } \mathcal{S}(s) \mathcal{F}(\mathcal{U}_{n})(t_{n} -s) ds +\int_{\lfloor t_{n} \rfloor}^{ t_{n} } S(s) \mathcal{F}(\mathcal{U}_{n})(t_{n} -s) ds, \nonumber \\[0,4cm]
		& = & \mathcal{S}(t_{n})\mathcal{U}_{n}(0) + \sum_{k=0}^{\lfloor t_{n} \rfloor -1}\int_{k}^{k+1 } \mathcal{S}(s) \mathcal{F}(\mathcal{U}_{n})(t_{n} -s) ds +\int_{\lfloor T_{n} \rfloor}^{ t_{n} } \mathcal{S}(s) \mathcal{F}(\mathcal{U}_{n})(t_{n} -s) ds, \nonumber\\[0,4cm]
		& = & \mathcal{S}(t_{n})\mathcal{U}_{n}(0) + \sum_{k=0}^{\lfloor t_{n} \rfloor -1} \mathcal{S}(k) \int_{0}^{1 } \mathcal{S}(s) \mathcal{F}(\mathcal{U}_{n})(t_{n} -s -k) ds +\int_{\lfloor t_{n} \rfloor}^{ t_{n} } \mathcal{S}(s) \mathcal{F}(\mathcal{U}_{n})(t_{n} -s) ds, \nonumber\\[0,4cm] 
		& = & \mathcal{S}(t_{n})\mathcal{U}_{n}(0) + \sum_{k=0}^{\lfloor t_{n} \rfloor -1} \mathcal{S}(k) I_{n,k} + I_{n} .
		\end{eqnarray}
		Let $\epsilon\in (0, \min (1-s,(5-p)/2,(17-3p)/14,1/2) )$.	Since $(\mathcal{S}(t))_{t\geq 0} $ is a semi-group in $\mathcal{H}^{\epsilon}$, then there exists $C>0$ such that 
		\begin{eqnarray}
		\Vert I_n \Vert_{\mathcal{H}^\epsilon } &=& \left \Vert \int_{\lfloor t_{n} \rfloor}^{ t_{n} } \mathcal{S}(s) \mathcal{F}(\mathcal{U}_{n})(t_{n} -s) ds \right \Vert_{\mathcal{H}^\epsilon}, \nonumber\\
		&\leq & \int_{\lfloor t_{n} \rfloor}^{ t_{n} } \Vert \mathcal{S}(s) \mathcal{F}(\mathcal{U}_{n})(t_{n} -s)\Vert_{\mathcal{H}^\epsilon} ds  , \nonumber\\
		&\leq & C \int_{\lfloor t_{n} \rfloor}^{ t_{n} } \Vert \mathcal{F}(\mathcal{U}_{n})(t_{n} -s) \Vert_{\mathcal{H} ^\epsilon }  ds  , \nonumber\\
		&\leq & C \int_{\lfloor t_{n} \rfloor}^{ t_{n} } \left( \Vert f_1(u_{n})(t_{n} -s) \Vert_{H ^\epsilon _0 } + 
		\Vert f_2(v_{n})(t_{n} -s) \Vert_{H ^\epsilon _0 } \right)  ds  . 
		\end{eqnarray}
		We can obtain similar bounds for  $I_{n,k}$ with different sets of integration. Using Lemma \ref{L5}, and  since the energy of $ \mathcal{U}_n $ is uniformly bounded, then  terms $  I_{n}  $, as well as $I_{n,k}$, are bounded by the same constant $M$ in $\mathcal{H}^\epsilon $ uniformly in $n$ and $k$. Moreover, we have
		\begin{eqnarray}
		\Vert \mathcal{U}_n (t_n) - \mathcal{S}(t_{n})\mathcal{U}_{n}(0) \Vert_{\mathcal{H}^\epsilon} & \leq &     \left \Vert \sum_{k=0}^{\lfloor t_{n} \rfloor -1} \mathcal{S}(k) I_{n,k} + I_{n} \right \Vert_{\mathcal{H}^\epsilon}, \nonumber \\
		& \leq &  \sum_{k=0}^{\lfloor t_{n} \rfloor -1} \Vert \mathcal{S}(k) I_{n,k} \Vert_{\mathcal{H}^\epsilon} +
		\Vert  I_{n} \Vert_{\mathcal{H}^\epsilon} . \nonumber 
		\end{eqnarray}
		Using Corollary \ref{cor1}, we get
		\begin{eqnarray}
		\Vert \mathcal{U}_n (t_n) - \mathcal{S}(t_{n})\mathcal{U}_{n}(0) \Vert_{\mathcal{H}^\epsilon} & \leq &  \sum_{k=0}^{\lfloor t_{n} \rfloor -1} e^{-\beta k} M + M , \nonumber \\
		& \leq &  M \left( \sum_{k=0}^{\lfloor t_{n} \rfloor -1} e^{-\beta k} +1 \right)  , \nonumber \\
		& \leq &  M \left( 1+ \dfrac{1}{1- e^{- \beta }}\right) . \nonumber 
		\end{eqnarray}
		In particular, denoting $R_{n}=\mathcal{U}_n (t_n) - \mathcal{S}(t_{n})\mathcal{U}_{n}(0)\in \mathcal{H}$, we prove that  $\sup_{n\in \mathbb{N}}\Vert R_{n}\Vert_{\mathcal{H}^\epsilon}<+\infty$. Using the  Rellich theorem, we can extract a subsequence so that $R_{\varphi(n)} \underset{n\to +\infty}{\longrightarrow }R_{\infty}$ in  $\mathcal{H} $ for some $R_{\infty}\in \mathcal{H}$.
		Moreover, since $(\mathcal{U}_{n,0})_{n \geq 0 }$ is bounded in $\mathcal{H}$ and $t_{n}\to +\infty$, Theorem \ref{th2} shows that $\mathcal{S}(t_{n})\mathcal{U}_{n}(0)  \underset{n\to +\infty}{\longrightarrow }0$ in  $\mathcal{H} $. In particular, $ \mathcal{U}_{\varphi(n)} (t_{\varphi(n)})  \underset{n\to +\infty}{\longrightarrow }R_{\infty}$ in  $\mathcal{H} $ and $R_{\infty}= \mathcal{U}_{\infty,0}$ by uniqueness of the weak limit. 
		
		$\mathcal{U}_{\infty}$ is defined as the solution of \eqref{system} on $[0,+\infty)$ with an initial datum $\mathcal{U}_{\infty}(0)=\mathcal{U}_{\infty,0}$ while $\mathcal{U}_{\phi (n)}(t_{\phi (n)}+.)$ is the unique solution of \eqref{system} on $[0,+\infty)$ with an initial datum $\mathcal{U}_{\phi (n)}(t_{\phi (n)})$. Then, since $\mathcal{U}_{\phi (n)}(t_{\phi (n)})\to \mathcal{U}_{\infty}(0)$ in $\mathcal{H}$, the local uniform continuity of the flow map gives that for all $T >0$, we have $\mathcal{U}_{\phi (n)}(t_{\phi (n)}+.)\longrightarrow \mathcal{U}_{\infty} (.)$ in $\mathcal{C}^0 ([0,T[, \mathcal{H})$. It is worth mentioning that the local uniform continuity of the flow map is proved in Theorem \ref{theo} for small $T$ depending on the norm in $\mathcal{H}$, but it is easy to iterate it on any interval $[0,T]$ since we have a priori bounds on the energy.
		
	\end{proof}
	
	Note that it could seem surprising in the previous proof that the Duhamel term is more regular, that is  in $\mathcal{H}^{\epsilon}$. It is a consequence of the fact that the nonlinearity is subcritical which was crucial in Lemma \ref{L5}.

	\begin{proposition}(Unique continuation)\label{L3}\item
		Let	$f_i \in \mathcal{C}^1(\mathbb{R},\mathbb{R})$, for $i=1,2$, satisfy \eqref{nonlinearity}. We assume that $f_i$ is analytic, $ \omega_{b} $ satisfies the geometric control condition and $ supp(b) \subset \omega_{a} $.
		The unique solution $ (u,v) $ in $  C (\mathbb{R}_{+},(H_{0} ^{1}( \Omega) )^{2}) \cap C ^{1}(\mathbb{R}_{+},(L ^{2}( \Omega) )^{2}) $ for the system
		\begin{eqnarray}\label{cont}
		\begin{cases}
		\partial _{tt} u -\Delta u  + b(x) \partial _t v + f_1(u) = 0 & in \; \Omega \times \mathbb{R}_{+} ^{*},\\
		\partial _{tt} v -\Delta v - b(x) \partial _t u + f_2(v)  = 0 & in \; \Omega \times \mathbb{R}_{+} ^{*} ,\\
		u=v=0  & on \; \Gamma \times \mathbb{R}_{+} ^{*},\\
		a(x) \partial _t u =0  & in \; \Omega \times \mathbb{R}_{+} ^{*},\\
		\left(  u_{0} , v_{0} ,  u_{1} , v_{1} \right) \in \mathcal{H},
		\end{cases}
		\end{eqnarray}
		is the trivial one $(u,v)= (0,0)$. \\
	\end{proposition}
	\begin{proof} 
		The fourth equation of \eqref{cont} gives $ a(x) \partial _t u =0 $ a.e in $ \Omega \times \mathbb{R}_{+} ^{*}$. Then $ \partial _t u =0 $ a.e. in $ \omega _{a} \times \mathbb{R}_{+} ^{*}$. 
		\begin{eqnarray}\label{uuu}
		u(x,t) = u(x)  \quad \textnormal{ a.e. for} \; (x,t) \in \omega _{a} \times \mathbb{R}_{+} ^{*}. 
		\end{eqnarray}
		We derive in the sense of distributions the first equation of the system \eqref{cont} relative to the variable time, and we get 
		$$ \partial _{ttt} u -\Delta \partial _t u  + b(x) \partial _{tt} v + \partial _t u f_1^\prime (u) = 0 .$$
		Thus, we have $ b(x) \partial _{tt}v = \partial _{tt}(bv)  =0 $ a.e. in $ \omega_a \times \mathbb{R}_{+} ^{*}$ in the distributional sense.  Consequently, there exists $g$ and $h\in L^{2}(\omega_{a})$ so that 
		
		\begin{eqnarray}\label{vvv}
		bv(x,t) = g(x) t + h(x) \quad \textnormal{ a.e. for} \; (x,t) \in \omega _{a} \times \mathbb{R}_{+} ^{*}.\nonumber
		\end{eqnarray}
		Using \eqref{equiv}, Poincar\'e inequality and the boundedness of $b$, we obtain for all $t\geq 0$
		\begin{eqnarray}
		\int_{t}^{t+1}E_{u,v}(\tau) d\tau & \geq & C \Vert (u,v,u^\prime , v ^\prime) \Vert_{L^{2}_{[t,t+1]}\mathcal{H}} ^2 \geq  C \Vert \nabla v \Vert ^2_{L^{2}_{[t,t+1]}L^{2}(\Omega)}  \geq  C \Vert  v \Vert_{L^{2}_{[t,t+1]}L^{2}(\Omega)} ^2 \geq  C \Vert  bv \Vert_{L^{2}_{[t,t+1]}L^{2}(\omega_{a})} ^2\nonumber \\
		& \geq & C \displaystyle \int_{t}^{t+1}\int_{\omega_a}  \vert g(x)\tau+ h(x) \vert ^2 dxd\tau\geq Ct^{2} \Vert  g \Vert^{2}_{{L^{2}(\omega_{a})}}-C\Vert  h \Vert^{2}_{{L^{2}(\omega_{a})}}.  
		\end{eqnarray}
		Since $\int_{t}^{t+1}E_{u,v}(\tau) d\tau$ is bounded uniformly for $t>0$, we get $g=0$. In particular, we have in the sense of distribution in $\omega_{a}$
		$$b \partial_{t}v (x,t)=0, \quad \forall (x,t) \in \omega_a \times \mathbb{R}_{+} ^*. $$ 
		Since $supp (b)\subset \omega_{a}$, we get the same result on $\Omega$.
		
		
		Finally, we get the two following non coupled systems
		\begin{eqnarray}\label{cuu}
		\begin{cases}
		\partial _{tt} u -\Delta u   + f_1(u)= 0 & in \; \Omega \times \mathbb{R}_{+} ^{*},\\
		\partial _t u =0  & in \; \omega _a \times \mathbb{R}_{+} ^{*},\\
		u=0  & on \; \Gamma \times \mathbb{R}_{+} ^{*},\\
		( u_{0} ,  u_{1} ) \in \left( H_{0} ^{1} ( \Omega ) \times L^{2} ( \Omega ) \right) ,
		\end{cases}
		\end{eqnarray}
		and
		\begin{eqnarray}\label{cuv}
		\begin{cases}
		\partial _{tt} v -\Delta v  + f_2(v) = 0 & in \; \Omega \times \mathbb{R}_{+} ^{*} ,\\
		\partial _t v =0  & in \; \omega_{b} \times \mathbb{R}_{+} ^{*},\\
		v=0  & on \; \Gamma \times \mathbb{R}_{+} ^{*},\\
		(  v_{0} , v_{1} ) \in  H_{0} ^{1} ( \Omega ) \times L^{2} ( \Omega ) .
		\end{cases}
		\end{eqnarray}
		Using the result of unique continuation given in \cite[Corollary 6.2]{2} (with suitable translation in time) and the fact that both $\omega_{a}$ and $\omega_{b}$ satisfy the geometric control condition, \eqref{cuu} gives that $u=0$ and \eqref{cuv} gives that $v=0$. Consequently, we have $ (u,v)=(0,0) $.
	\end{proof}
	\begin{proof}[Proof of Theorem \ref{th3.6}]
		We argue by contradiction. We suppose that inequality \eqref{IO} is false for all $T>0$. Then, there exists $ \mathcal{U}_n =(u_n,v_n,\partial _t u_n , \partial _t v_n)$ which represents a sequence of solution of \eqref{system} and a sequence of time, where $ T_{n} \underset{n\to \infty}{\longrightarrow} \infty $ such that
		\begin{eqnarray}
		\begin{cases}
		E_{u_n ,v_n}(0) \leq E_0 , \\[0.3cm]
		\displaystyle\int _{0} ^{T_{n}} \displaystyle\int _{\Omega } a(x) \vert \partial _t  u_{n}  \vert^2 dx dt \leq \frac{1}{n} E_{u_n ,v_n} (0) .
		\end{cases} \label{contradobserv}
		\end{eqnarray}
		We note $ \alpha _{n} = ( E _{u_n , v_n} (0))^{\frac{1}{2}} $.  Since $ \alpha _{n} \in [ 0, \sqrt{E_{0}} ] $, for all $n \in \mathbb{N}$, then we can extract a subsequence that will also be noted by $ \alpha _{n} $, which converges. We note its limit  $ \alpha$, then we have $ \alpha \in [ 0, \sqrt{E_{0}} ]$.\\
		Here, we distinguish two cases:
		
		\begin{enumerate}
			\item \textbf{Case} $\mathbf{\alpha \neq 0} $\\
			Using \eqref{equiv}, we obtain for all $ t>0 $
			\begin{eqnarray}
			\Vert  \mathcal{U}_{n} (t) \Vert_{\mathcal{H}} ^{2}	& \leq & C E_{u_n ,v_n}(t) \lesssim  E_0 .
			\end{eqnarray}
			Then, sequence  $ ( \Vert  \mathcal{U}_{n} (t) \Vert_{\mathcal{H}} )_{n \in \mathbb{N}} $ is uniformly bounded for $t>0$.\\
			We set $ (w_{n,1}, w_{n,2})(.) = (u_{n},v_{n}) (T_{n}/2+.)$.  Proposition \ref{L2} ensures the existence of a subsequence  of $ (w_{n,1}, w_{n,2})_{n\geq 0} $, which will be denoted also $ (w_{n,1}, w_{n,2}) $, and a global solution $ (w_{1}, w_{2}) $ on $[0,+\infty)$ of system \eqref{system} such that for all $T>0$ we have
			\begin{eqnarray}
			\mathcal{W}_{n}  =(w_{n,1}, w_{n,2} , \partial _t w_{n,1},\partial _t w_{n,2} ) \underset{n\to +\infty}{\longrightarrow} (w_{1}, w_{2} , \partial _t w_{1},\partial _t w_{2} )=\mathcal{W} \; \;in \; C^{0} ([0,T],\mathcal{H}).\nonumber
			\end{eqnarray}
			We have
			\begin{eqnarray} \label{123}
			E_{w_{n,1} ,w_{n,2}}(0)= E_{u_n ,v_n} (T_n /2) \leq E_{u_n ,v_n} (0),
			\end{eqnarray}
			and 
			\begin{eqnarray} \label{1234}
			E_{u_n ,v_n} (T_n /2) &=& E_{u_n ,v_n} (0) - \int _{0} ^{T_n /2} \int _{\Omega } a(x) \vert \partial _t u_n  (x,t) \vert ^2 dx dt , \nonumber \\
			&\geq & E_{u_n ,v_n} (0) - \frac{1}{n} E_{u_n ,v_n} (0), \nonumber \\
			&\geq & (1 - \frac{1}{n}) E_{u_n ,v_n} (0) .
			\end{eqnarray}
			We pass to the limit in equations \eqref{123} and \eqref{1234}, we obtain $ E_\mathcal{W} = \alpha^2 >0 $.\\
			On the other hand, \eqref{contradobserv} ensures that $ a \partial _t w_{n,1}  $ converges to $0$ in $ L^2 \left( \left[ -\frac{T_n}{2},\frac{T_n}{2}\right] ,L^2(\Omega)\right) $ and therefore on each $ L^2 \left( \left[ 0,T \right] ,L^2(\Omega)\right) $. This implies that $ a \partial _t w_{1} \equiv 0 $ on $[0,T]\times \Omega$ for any $T>0$ and thus $\mathcal{W}$ is a global solution of  system \eqref{cont}. Consequently, from Proposition \ref{L3}, we obtain $\mathcal{W} \equiv 0 $, and then  $E_{w_1 ,w_2} (0)= 0$, which contradicts $E_{w_1 ,w_2} (0) = \alpha ^2 >0$.  
			\item \textbf{Case} $\mathbf{\alpha = 0} $\\
			The assumptions on $f_i$, for $i=1,2$, allow to write $f_i(s) = f_i ^\prime (0) s +R_i(s)$ with 
			\begin{eqnarray}\label{rest}
			\vert R_i(s) \vert \leq C (\vert s \vert ^2 +\vert s \vert ^p) \quad \text{and } \quad \vert R_i ^\prime (s) \vert \leq C (\vert s \vert  +\vert s \vert ^{p-1}).
			\end{eqnarray} 
			We pose $ (w_{n,1}, w_{n,2}) = (u_{n}/\alpha_{n},v_{n}/\alpha_{n})$. Then $ (w_{n,1}, w_{n,2}) $ is the solution of system
			\begin{eqnarray}\label{11}
			\begin{cases}
			\partial _{tt} w_{n,1} -\Delta w_{n,1} + a(x) \partial _t  w_{n,1} + b(x) \partial _t w_{n,2} + f_1 ^\prime (0) w_{n,1} +\dfrac{1}{\alpha_{n} } R_1 (\alpha_{n} w_{n,1}) = 0 & in \; \Omega \times \mathbb{R}_{+} ^{*},\\
			\partial _{tt}w_{n,2} -\Delta w_{n,2} - b(x) \partial _t w_{n,1} + f_2 ^\prime (0) w_{n,1} +\dfrac{1}{\alpha_{n} } R _2(\alpha_{n} w_{n,2})   = 0 & in \; \Omega \times \mathbb{R}_{+} ^{*} ,
			\end{cases} 
			\end{eqnarray}
			we also have
			\begin{eqnarray}
			\int _{0} ^{T_n } \int _{\Omega } a(x) \vert \partial _{t} w_{n,1}  (x,t) \vert ^2 dx dt \leq \dfrac{C}{\alpha_{n} ^2 n} E_{u_n ,v_n}  (0) \leq \dfrac{C}{n}.
			\end{eqnarray}
			Then \eqref{equiv} ensures that the sequence $\left( \mathcal{W}_n(t)=(w_{n,1}, w_{n,2},\partial _t w_{n,1},\partial _t  w_{n,2})(t)\right) _{n \in \mathbb{N}}$ is bounded uniformly in $\mathcal{H}$.  More precisely, we have for all  $n \in \mathbb{N}$ and $t\in \left[ 0, T_n \right] $
			\begin{eqnarray}
			\label{bornesupW}
			\Vert \mathcal{W}_{n}(t) \Vert _{\mathcal{H}}^2 = \dfrac{\Vert \mathcal{U}_{n} ( t)\Vert _{\mathcal{H}} ^2}{\alpha_{n}^2} & \leq& C\dfrac{E_{u_n ,v_n}  (t)}{\alpha_{n}^2} \leq C \dfrac{E_{u_n ,v_n}  (0)}{\alpha_{n}^2}=C,
			\end{eqnarray}
			and the lower bounds,		
			
			\begin{eqnarray}\label{321}
			\Vert \mathcal{W}_{n}( t) \Vert _{\mathcal{H}}^2 = \dfrac{\Vert \mathcal{U}_{n} (t)\Vert _{\mathcal{H}} ^2}{\alpha_{n}^2} & \geq & \dfrac{E_{u_n ,v_n}  (t)}{C\alpha_{n}^2}\geq  \dfrac{E_{u_n ,v_n}  (T_{n})}{C\alpha_{n}^2}, \nonumber \\
			& \geq &  \dfrac{1}{C\alpha_{n}^2} \left( E_{u_n ,v_n}  (0) - \int _{0} ^{T_{n} } \int _{\Omega } a(x) \vert \partial _t u_n  (x, s) \vert ^2 dx d s \right) , \nonumber \\
			& \geq &  \dfrac{1}{C\alpha_{n}^2} \left( E_{u_n ,v_n}  (0) - \dfrac{1}{n} E_{u_n ,v_n}  (0) \right) , \nonumber \\
			& \geq &  \dfrac{1}{C\alpha_{n}^2} \left( 1 - \dfrac{1}{n}  \right) E_{u_n ,v_n}  (0) = \frac{1}{C}\left( 1-\dfrac{1}{n} \right) .\nonumber 
			\end{eqnarray}
			Then for large $n$ and all $t\in \left[ 0, T_n \right] $, we obtain 
			\begin{eqnarray}
			\Vert \mathcal{W}_{n}( t) \Vert _{\mathcal{H}}^2& \geq & \dfrac{1}{2C} > 0.
			\end{eqnarray}
			On the other hand, we have $ (w_{n,1}, w_{n,2})$ is the solution of  system \eqref{11}  with a nonlinearity satisfying
			\begin{eqnarray} \label{estimRn}\left \vert \dfrac{1}{\alpha_{n} } R_i (\alpha_{n} s) \right \vert \leq C (\alpha_{n}\vert s \vert^{2}  +\alpha_{n} ^{p-1}\vert s \vert ^p)\end{eqnarray} thanks to \eqref{rest}. 
			In particular, combined with \eqref{2.24}, this gives, recalling $3<p<5$ 
			\begin{eqnarray}\label{estimRi}
			\left\Vert \dfrac{1}{\alpha_{n}}R_1(\alpha_{n} w_{n,1}) \right\Vert _{L^1 ((k,k+1),L^2)} & \leq & \displaystyle \int_{k}^{k+1} \left\Vert \dfrac{1}{\alpha_{n}}R_1(\alpha_{n} w_{n,i}) \right\Vert _{L^2} ds \nonumber \\
			& \leq & C \displaystyle \int_{k}^{k+1}   \left( \alpha_{n}\Vert w_{n,1} \Vert _{L^4}^{2}  +\alpha_{n} ^{p-1}\Vert w_{n,1} \Vert _{L^{2p}} ^p\right)  ds \nonumber\\
			& \leq &  C \displaystyle \int_{k}^{k+1}   \left( \alpha_{n}\Vert w_{n,1} \Vert _{L^{2p}} ^p  +\alpha_{n} ^{p-1}\Vert w_{n,1} \Vert _{L^{2p}} ^p\right)  ds  \qquad\qquad (p>2)\nonumber\\
			&\leq &C\alpha_{n}\left\Vert w_{n,1}\right\Vert_{L^{\frac{2p}{p-3}} ((k,k+1),L^{2p})}  +C\alpha_{n} ^{p-1}\left\Vert w_{n,1}\right\Vert _{L^{\frac{2p}{p-3}} ((k,k+1),L^{2p})} ^{p}\nonumber 
			\end{eqnarray}
			Since $w_{n}$ is a solution of \eqref{11} and applying the Strichartz estimate (still valid with the additional linear term), we get for $t\in [k,k+1]$, uniformly in $k$, with $n$ so that $T_{n}\geq k$, 
			\begin{eqnarray}
			\left\Vert w_{n}\right\Vert _{L^{\frac{2p}{p-3}} ((k,t),L^{2p})} &\leq &\Vert  \mathcal{W}_{n}(k)  \Vert _{\mathcal{H}}+C\alpha_{n} 	\left\Vert w_{n}\right\Vert _{L^{\frac{2p}{p-3}} ((k,t),L^{2p})}^{2}+C \alpha_{n}^{p-1} \left\Vert w_{n}\right\Vert _{L^{\frac{2p}{p-3}} ((k,k+1),L^{2p})}^{p}\nonumber\\
			&\leq&C+C\alpha_{n} 	\left\Vert w_{n}\right\Vert _{L^{\frac{2p}{p-3}} ((k,t),L^{2p})}^{2}+C \alpha_{n}^{p-1} \left\Vert w_{n}\right\Vert _{L^{\frac{2p}{p-3}} ((k,t),L^{2p})}^{p}	,
			\end{eqnarray}		
			where  \eqref{bornesupW} is used and $\left\Vert w_{n}\right\Vert _{L^{q} ((k,t),L^{r})}$ is conventionally the sum of the norms of $w_{n,1}$ and $w_{n,2}$. A bootstrap argument states that for $\alpha_{n}$ that is small enough, \\$	\left\Vert w_{n}\right\Vert _{L^{\frac{2p}{p-3}} ((k,k+1),L^{2p})}\leq C$, with $n$ large enough and uniformly in $k\leq T_{n}-1$.
			Applying again Strichartz estimates (estimate \eqref{strichartz}), we obtain
			\begin{eqnarray}
			\Vert w_{n,1} \Vert _{L^q ((k,k+1),L^r)} + \Vert w_{n,2} \Vert _{L^q ((k,k+1),L^r)} \leq C , \nonumber 
			\end{eqnarray}
			for any admissible couple $(q,r)$. Thus, we have 
			$$ \frac{1}{\alpha_{n}} \left( 	\Vert u_{n} \Vert _{L^q ((k,k+1),L^r)} + \Vert v_{n} \Vert _{L^q ((k,k+1),L^r)} \right) \leq C . $$ 
			Consequently, $ \Vert u_{n} \Vert _{L^q ((k,k+1),L^r)} + \Vert v_{n} \Vert _{L^q ((k,k+1),L^r)} \leq C \alpha_{n} $ uniformly for $n,k \in \mathbb{N}$.  
			Getting back to \eqref{estimRi}, and using our bound, we have now for $n$ that is large enough and $k\leq T_{n}$
			\begin{eqnarray}
			\left\Vert \dfrac{1}{\alpha_{n}}R_1(\alpha_{n} w_{n,1}) \right\Vert _{L^1 ((k,k+1),L^2)} & \leq & C\alpha_{n}+C\alpha_{n} ^{p-1}\leq 2C\alpha_{n}.\nonumber 
			\end{eqnarray}
			and the same holds for $w_{n,2}$ and $R_{2}$.
		\end{enumerate} 
		Furthermore, using the Duhamel formula, we get for all $t>0$ 
		$$ \mathcal{W}_n (t) =\tilde{ \mathcal{S}}(t)\mathcal{W}_n(0) + \int_{0}^{t} \mathcal{S}(t-s) G(\mathcal{W}_n)(s)ds, $$
		where $ G(\mathcal{W}_n) = \left( 0,0, -\dfrac{1}{\alpha_{n} } R_1 (\alpha_{n} w_{n,1}) ,- \dfrac{1}{\alpha_{n} } R_2 (\alpha_{n} w_{n,2}) \right) $ and $(\tilde{ \mathcal{S}}(t))_{t\geq 0}$ is the semigroup generated by the operator $\tilde{ \mathcal{A}}$ given by 
		
		$$  \tilde{ \mathcal{A}}=
		\begin{pmatrix}
		0 & 0 & I & 0 \\
		0 & 0 & 0 & I \\
		\Delta - f_1(0) & 0 &  -a(x) & -b(x) \\
		0& \Delta  - f_2(0) & b(x)  & 0
		\end{pmatrix}.
		$$

		We can argue as in Proposition \ref{L2} and write 
		\begin{eqnarray}
		\mathcal{W}_n (T_n) &= &\tilde{ \mathcal{S}} (T_n)W_n(0) + \int_{0}^{T_n} \tilde{ \mathcal{S}}(T_n-s) G(\mathcal{W}_n)(s)ds ,\nonumber \\
		&= &\tilde{ \mathcal{S}}(T_n)\mathcal{W}_n(0) +  \sum_{k=0}^{\lfloor T_{n} \rfloor -1} \tilde{ \mathcal{S}}(T_n-k) \int_{0}^{1 }  \tilde{ \mathcal{S}}(-y) G(\mathcal{W}_n)(y+k)dy , \nonumber \\
		&&+ \tilde{ \mathcal{S}}(T_n-\lfloor T_{n} \rfloor) \int_{0}^{ T_{n}-\lfloor T_{n} \rfloor } \tilde{ \mathcal{S}}(-y) G(\mathcal{W}_n)(y+\lfloor T_{n} \rfloor)ds. \nonumber 
		\end{eqnarray}
		Since $\Omega $ is a bounded domain, $f_i ^\prime (0) \geq 0 $, for $i=1,2$, and due to the exponential stability  of the semi-group $(\mathcal{S}(t))_{t\geq 0}$, the semi-group $(\tilde{ \mathcal{S}}(t))_{t\geq 0}$ is also exponentially stable. Furthermore, we get
		$$ \Vert \mathcal{W}_n (T_n) \Vert_{ \mathcal{H}} \leq C e^{-\beta t} + C \alpha_{n}.$$
		Finally, we obtain
		$$ \Vert \mathcal{W}_n (T_n) \Vert_{ \mathcal{H}} \longrightarrow 0 ,$$
		which is in contradiction with \eqref{321}.
		
		This concludes the proof of Theorem \ref{th3.6} and therefore of Theorem \ref{th1}.
	\end{proof}		
	\bigskip		
	
	\textbf{Acknowledgments:} We are also very grateful to the anonymous reviewers for their careful reading of the article and remarks, which have allowed us to significantly improve the presentation of the paper. A significant part of the work on this paper was carried out while Radhia Ayechi visited the Camille Laurent at Sorbonne Universit\'e from January to june 2018. Radhia Ayechi gratefully acknowledges the financial support from \textbf{Erasmus +} during this visit. 
\bibliographystyle{plain}

\bibliography{semilinear}
\nocite{*}
\medskip
\medskip

\end{document}